\documentclass[11pt, reqno]{amsart}

\usepackage{amssymb,graphicx}
\usepackage[latin1]{inputenc}
\usepackage[T1]{fontenc}
\usepackage[english]{babel}
\usepackage[all]{xy}
\usepackage{graphics}
\usepackage{graphicx}
\usepackage{verbatim}
\usepackage{color}
\usepackage{url}
\newtheorem{thm}{Theorem}[section]
\newtheorem{prop}[thm]{Proposition}

\newtheorem{lem}[thm]{Lemma}
\theoremstyle{definition}
\newtheorem{defin}[thm]{Definition}
\newtheorem{rem}[thm]{Remark}
\newtheorem*{thm*}{Theorem}

\newcommand{\weg}[1]{}
 
  \newcommand{\diag}{\mathrm{diag}}
    
  \newcommand\bib[1]{\bibitem[#1]{#1}}
  \newcommand{\ms}{\eta}
\newcommand{\sms}{\tilde\ms}
\newcommand{\vol}{{\mathrm{vol}}}
\newcommand{\Ric}{\mathrm{Ric}}
\newcommand{\Rho}{{\mathsf P}}
  
   \def\CP{{\mathbb{CP}}}
  \def\C{{\mathbb C}}
\def\N{{\mathbb N}}

\def\R{{\mathbb R}}
\def\Z{{\mathbb Z}}

\title[C-projective transformations of complete K\"ahler manifolds]{On the groups of c-projective transformations of complete K\"ahler manifolds}
 \author{Vladimir S.  Matveev and Katharina Neusser}
\address{Institut f\"ur Mathematik,
Fakult\"at f\"ur Mathematik und Informatik,\newline
Friedrich-Schiller-Universit\"at Jena,
07737 Jena, Germany}
\email{vladimir.matveev@uni-jena.de}
\address{Mathematical Institute, Faculty of Mathematics and Physics, \newline Charles University, 186 75 Prague, Czech Republic}
\email{kath.neusser@gmail.com}
\subjclass{32Q15, 32J27, 53A20, 53C24, 22F50, 37J35}

\thanks{The paper was started during the visit of K.N. to Jena supported by DAAD Ostpartnerschaft-Programm and DFG GRK 1523. K.N. was also supported by GACR P201/12/G028.}
 
\begin{document}

\begin{abstract} 
We show that for any complete connected K\"ahler manifold, the index of the group of complex affine transformations 
in the group of c-projective transformations is at most two unless the K\"ahler manifold is isometric to complex projective space equipped with a positive constant multiple of the Fubini--Study metric. 
This establishes a stronger version of the recently proved Yano--Obata conjecture for complete K\"ahler manifolds.
\end{abstract} 
\maketitle

 \section{Introduction}\label{Introduction}
 Suppose $(M, J)$ is a complex manifold of real dimension $2n\geq 4$ and $g$ a K\"ahler metric on $(M,J)$ with Levi-Civita connection $\nabla$.
 A regular curve $\gamma: I \to M$ defined on some interval $I\subseteq \R$ is said to be \emph{$J$-planar} with respect to $\nabla$ (or $g$), if there exist functions $\alpha,\beta:I\rightarrow\mathbb{R}$ such that 
\begin{equation}
\label{eq:J-planar} 
\nabla_{\gamma'(t)}\gamma'(t)=\alpha \gamma'(t)+\beta J(\gamma'(t))\mbox{ for all } t\in I.
\end{equation} 

It follows from the definition that the property of being $J$-planar for a curve is independent of the parameterisation of the curve and that geodesics of $\nabla$ are $J$-planar curves.   
The $J$-planar curves form however a much larger family of curves than the family of geodesics---at every point and in every direction there exist 
infinitely many geometrically different  $J$-planar curves. Two K\"ahler metrics on $(M,J)$ are called \emph{c-projectively equivalent}, if they have the same $J$-planar curves, and
a \emph{c-projective transformation} of a K\"ahler manifold $(M,J, g)$ is a complex diffeomorphism of $M$ mapping $J$-planar curves to $J$-planar curves.

C-projective equivalence of K\"ahler metrics was first introduced in \cite{OT} and  
provided a prominent research direction in the Japanese and Soviet schools of
differential geometry, see e.g.\,\cite{Mik,S,Y}. Later, it was rediscovered under different names and with different
motivations. In particular, c-projectively equivalent metrics on a given K\"ahler manifold are essentially the same as  \emph{Hamiltonian $2$-forms} \cite{ACG, CMR}, and in dimension $\ge 6$ there are also essentially the
same as \emph{conformal Killing (or twistor) $(1,1)$-forms} (see \cite[App.~A]{ACG} or also
\cite[\S{1.3}]{MR}), and they are closely related 
to the so-called \emph{K\"ahler--Liouville integrable
systems of type $A$}, see e.g.~\cite{KT}. For an overview of the current developments and the renewed interest in c-projective geometry see also \cite{CEMN}.
 
 Given a K\"ahler manifold $(M,J,g)$ we shall write 
\begin{equation*}
 \textrm{Iso}(J, g)\subseteq \textrm{Aff}(J, g)\subseteq \textrm{CProj}(J,g)
 \end{equation*}
 for the groups of complex isometries, of complex affine transformations (complex diffeomorphisms of $M$ preserving the Levi-Civita connection) and of c-projective transformations respectively. 
 
For the Fubini--Study metric $g_{FS}$ on complex projective space $\CP^n$ it is well-known that $\textrm{Iso}(J, g_{FS})=\textrm{Aff}(J, g_{FS})$ and also that $\textrm{Aff}(J,g_{FS})$ is a proper subgroup of $\textrm{CProj}(J,g_{FS})$.
Indeed, the $J$-planar curves on $\CP^n$ (with respect to $g_{FS}$) are precisely those smooth regular
curves that lie within complex lines (see e.g.\,\cite[Ex.1]{MR}) and hence $\textrm{CProj}(J,g_{FS})$ can be identified with the complex projective linear group $\textrm{PGL}(n+1,\C)\cong\textrm{PSL}(n+1,\C)$.
Note that an element in
$\textrm{GL}(n+1,\C)$ induces a complex isometry of $g_{FS}$ if and only if it is proportional to a unitary isomorphism of
$\C^{n+1}$, which shows that $\textrm{Isom}(J, g_{FS})$ can be
identified with the Lie group
$\textrm{PU}(n+1)=\textrm{U}(n+1)/\textrm{U}(1)$ of projective unitary
transformations. Hence, $\textrm{CProj}(J,g_{FS})/\textrm{Aff}(J, g_{FS})$ has  infinitely many elements. In \cite{CEMN} it was however recently shown that $g_{FS}$ (up to multiplications by positive constants and isometries) is 
the only complete K\"ahler metric for which $\textrm{Aff}_0(J, g_{FS})$ is a proper subgroup of
$\textrm{CProj}_0(J,g_{FS})$, where the subscript $0$ denotes the connected components of the identity of the groups. This has answered affirmatively the so-called \emph{Yano--Obata conjecture for complete K\"ahler manifolds}---a metric c-projective analogue of the \emph{projective} and \emph{conformal Lichnerowicz conjectures} (see \cite{M1,M2, KM, BMR} respectively \cite{Fer1, Fer2, Ob,Sch, F}):
 
 \begin{thm}{\cite[Theorem 7.6]{CEMN}}  \label{thm:obata}
Let $(M, g, J)$ be a complete connected K\"ahler manifold of real dimension
$2n\geq 4$. Then, $\emph{Aff}_0(J, g) = \emph{CProj}_0(J, g)$ unless $(M, g,
J)$ has constant positive holomorphic sectional curvature.
\end{thm}
In the compact case Theorem \ref{thm:obata} was first proved (using different methods and crucially compactness) in \cite{FKMR, MR} and also generalised to the pseudo-K\"ahler setting in \cite{BMR}.

 Hence, by Theorem \ref{thm:obata}, there are no flows of non-affine c-projective transformations on a connected complete K\"ahler manifold unless it is isometric to
$(\CP^n, J, cg_{FS})$ for some positive constant $c\in \R$. There are nevertheless examples of complete K\"ahler metrics other than positive constant multiples of the Fubini--Study metric for which 
 $\textrm{Aff}(J, g)$ is still a proper subgroup of $\textrm{CProj}(J, g)$. These  examples can be constructed using a similar idea as in \cite[\S 1.3]{M2} and in these examples the index of $\textrm{Aff}(J, g)$ in $\textrm{CProj}(J, g)$ 
 is two. The aim of this paper is to show the following theorem:
 
 \begin{thm}\label{main_thm}
 Suppose $(M, J, g)$ is a connected complete K\"ahler manifold of real dimension $2n\geq 4$ whose holomorphic sectional curvature is not a positive constant. Then, the index 
 of the subgroup $\emph{Aff}(J,g)$ in the group $\emph{CProj}(J, g)$ is at most $2$.
 \end{thm}
 
 As we will explain in Section \ref{proof_thm2},  as  a consequence we will also obtain:
 \begin{thm}\label{main_thm2}
 Suppose $(M, J, g)$ is a connected complete K\"ahler manifold of real dimension $2n\geq 4$ whose holomorphic sectional curvature is not a nonnegative constant. If $\emph{Aff}(J,g)\subsetneq \emph{CProj}(J,g)$, then the following statements hold:
  \begin{itemize}
 \item $\emph{Isom}(J,g)=\emph{Aff}(J,g)$
 \item $\emph{Isom}(J,g)$ has index $2$ in  $\emph{CProj}(J,g)$.
  \end{itemize}
 \end{thm}
 Let us also remark that the group of complex affine transformations of a complete connected K\"ahler manifold $(M, J, g)$ is well understood: the universal cover of $(M, J, g)$ decomposes according to the de Rham decomposition 
 into a product of K\"ahler manifolds
 $$(M_0, J_0, g_0)\times (M_1, J_1, g_1) \times...\times (M_k, J_k, g_k),$$ where $(M_0, J_0, g_0)$ is complex Euclidean space and $(M_i, J_i, g_i)$ for $1\leq i\leq k$ a complete simply-connected K\"ahler manifold with irreducible holonomy.
 Any complex affine transformation preserves the flat factor $(M_0, J_0, g_0)$ and acts as a complex affine transformation on it, and permutes the factors $M_i$ for $1\leq i\leq k$. Moreover, for $i =1,..., k$  one has that
$\phi|_{M_i}$ is an isometry if $\phi(M_i)=M_i$,  and a homothety (or an isometry) onto its image otherwise, see for example \cite[Chapter IV]{L}.
  
  \subsection{Structure of the proof of Theorem \ref{main_thm} and relation to previous results}\label{structure_section}
 An important role in the proof of Theorem \ref{main_thm} will be played by the metrisability equation \eqref{metrisability_equation} on a K\"ahler manifold, which 
 is a c-projectively invariant linear overdetermined system of PDEs of finite type; we shall recall its definition and key properties in Section \ref{C-prj.-Background}. The set of c-projectively equivalent metrics on a K\"ahler manifold $(M,J,g)$ embeds as an open subset into the vector space $\mathcal S$ of its solutions and the dimension of $\mathcal S$ is 
called the {\it degree of mobility} of  $(M,J,g)$.  
 
 Since for a complete K\"ahler metric $g$ and $\phi\in\textrm{CProj}(J,g)$, the pull-back $\phi^*g$ is a c-projectively equivalent complete
 K\"ahler metric,  \cite[Corollary 7.3]{CEMN} implies:
  \begin{thm}\label{deg_at_least_3} 
 Suppose $(M, J, g)$ is a connected complete K\"ahler manifold of real dimension $2n\geq 4$ whose holomorphic sectional curvature is not a positive constant.
 If the degree of mobility of $(M, J, g)$ is at least $3$, then $\emph{Aff}(J,g)=\emph{CProj}(J, g)$.
 \end{thm}
 Note moreover that on a K\"ahler manifold $(M,J, g)$ of degree of mobility $1$ any c-projectively equivalent metric is necessarily a nonzero constant multiple of $g$. Hence, in this case any c-projective transformation is necessarily a homothety and one has $\textrm{Aff}(J,g)=\textrm{CProj}(J, g)$. Thus, in view of Theorem \ref{deg_at_least_3},
it remains to prove Theorem \ref{main_thm} under the assumption of degree of mobility $2$.  

Suppose now $(M,J,g)$ satisfies the assumptions of Theorem \ref{main_thm} and is of degree of mobility $2$.  
 Since the metrisability equation is c-projectively invariant, the group of c-projective transformations of $g$ naturally acts on its solution space defining a representation of 
 $\textrm{CProj}(J,g)$ on the $2$-dimensional vector space $\mathcal S$. 
 In Section \ref{Aut_group}, based on a circle of ideas also used in \cite{CEMN, Z}, we will analyse this representation. In Proposition \ref{non-aff_c-proj_trans_mobility2} we show that, if $\phi\in\textrm{CProj}(J,g)\setminus \textrm{Aff}(J,g)$ acts on $\mathcal S$ as an isomorphism with positive determinant, then this isomorphism is necessarily diagonalisable with two distinct positive eigenvalues. 
Our investigations in Sections \ref{special_sol} and \ref{Section_Proof_Crucial_Prop} then show that this however can never be the case, that is, 
$\phi\in\textrm{CProj}(J,g)\setminus \textrm{Aff}(J,g)$ necessarily acts on $\mathcal S$ as an isomorphism with negative determinant. This in turn implies that the index of 
$\textrm{Aff}(J,g)$ in $\textrm{CProj}(J, g)$ is at most $2$ (see Proposition \ref{crucial_prop}), which hence establishes Theorem \ref{main_thm}.
 
Let us give some additional comments on the proof of Theorem \ref{main_thm} respectively Proposition \ref{crucial_prop} and its relation to previous results. Zeghib showed in \cite[Theorem1.3]{Z} that for compact Riemannian manifolds $(M, g)$ other than finite quotients of the standard sphere ($\dim(M)\geq 2$) the group of affine transformations 
$\textrm{Aff}(g)$ has at most index $2\dim(M)$ in the group $\textrm{Proj}(g)$ of projective transformations (diffeomorphisms of $M$ sending  geodesics to geodesics). The first author improved this result in \cite{M3} showing that the index is in fact at most two, and also generalised it in \cite{M4} to complete Riemannian manifolds.  Let us remark that Zeghib also claimed in \cite[\S 1.2]{Z}---without proof though, that one can show analogously to the proof of his projective result \cite[Theorem 1.3]{Z} that on compact K\"ahler manifolds other than complex projective space equipped with (a positive constant multiple of) the Fubini--Study metric the index of 
$\textrm{Aff}(J,g)$ in $\textrm{CProj}(J,g)$ is finite.

The ideas of \cite{Z, M3, M4} are also used in the proof of our Theorem \ref{main_thm}. Let us emphasise that the proof of our Theorem \ref{main_thm} is however not a straightforward application of the methods of \cite{CEMN} and the c-projective analogues of the techniques in \cite{Z,M3,M4}. To establish 
Proposition \ref{crucial_prop} we first proceed similar as in the proof of the Yano--Obata conjecture in \cite[Theorem 7.6]{CEMN}, which allows to reduce our considerations to a very special case (see Lemma \ref{inequ_1}).
To handle this special case however --- a step not needed in the projective case (see \cite{M4,Z}), new arguments are needed, which we develop in Sections \ref{special_sol} and \ref{Section_Proof_Crucial_Prop}. 
\newline\newline
\textbf{Acknowledgements}
\\We would like to thank the referee for helpful comments and suggestions leading to improvements of our article.
\section{C-projective structures}\label{C-prj.-Section}
In this section we recall some background on c-projective structures and review some properties of the geodesic flows of K\"ahler manifolds that 
admit c-projectively equivalent K\"ahler metrics; for details we refer to \cite{CEMN} and the references therein.
 
 \subsection{Notations}\label{C-prj.-Background}
  Suppose $(M, J)$ is a complex manifold of real dimension $2n\geq 4$. When it is convenient, we will use standard abstract index notation for tensors on $(M, J)$.
 To avoid any confusion with the notation in \cite{CEMN}, let us emphasise that in \cite{CEMN} the Greek alphabet is used to index real tensors 
 on $(M,J)$, whereas the Roman alphabet is used to index complex tensors. Since we will in this paper only work inside the real setting, we will use Roman indices for tensors 
 following the usual conventions for tensors on manifolds. 
 
 We call a linear connection $\nabla$ on $TM$ \emph{complex}, if $\nabla J=0$. Two complex linear connections $\widehat \nabla$ and $\nabla$ on $TM$ are called \emph{c-projectively-equivalent}, 
 if they have the same $J$-planar curves (note that \eqref{eq:J-planar} is well-defined for any connection). If both connections are torsion-free this is known to be equivalent \cite{OT} to the existence of a $1$-form $\Upsilon_a\in\Gamma(T^*M)$ such that 
 \begin{equation}\label{cproj_change}
\widehat\nabla_a X^b=
\nabla_a X^b+ \tfrac{1}{2}(\Upsilon_aX^b+\delta_{a}{}^b\Upsilon_cX^c-J_{a}{}^c\Upsilon_c J_{d}{}^bX^d-J_{a}{}^b\Upsilon_c J_d{}^cX^d).
\end{equation}

\begin{defin} A \emph{c-projective structure} on $(M,J)$ is an equivalence class $[\nabla]$ of c-projectively equivalent torsion-free complex linear connections on $TM$.
A \emph{c-projective transformation} of a c-projective manifold $(M, J, [\nabla])$ is a complex diffeomorphism of $M$ preserving $[\nabla]$, equivalently mapping $J$-planar curves to $J$-planar curves.
 \end{defin}
 
 Recall that on a complex manifold the real line bundle $\Lambda^{2n} T^*M$ is canonically oriented and hence one can form an $(n+1)$st positive root of this bundle:
 $$L:= (\Lambda^{2n}T^*M)^{\frac{1}{n+1}}\quad\textrm{ such that }\quad L^{\otimes^{n+1}}=\Lambda^{2n}T^*M.$$ 
 Specifically,  if $\{(U_\alpha, u_{\alpha})\}_{\alpha\in I}$ is an oriented atlas of $M$, then the cocycle of transition functions $U_{\alpha}\cap U_{\beta}\rightarrow \textrm{GL}(1,\R)$ defining 
 $L$ is given by 
\begin{equation*} 
x\mapsto \det \left(T_{u_\beta(x)}(u_\alpha\circ u_{\beta}^{-1})\right)^{-\frac{1}{n+1}}.
 \end{equation*}
 
 Since the line bundles $\Lambda^{2n} T^*M$ and $L$ are canonically oriented, they are trivialisable by the choice of a positive section, but there is no canonical trivialisation 
 of these bundles on a complex or c-projective manifold. In the sequel we also use for $L$ or $L^*$-valued tensors abstract index notation, e.g. $X^a$ may denote simply a vector field 
 or a section of $TM\otimes L$ respectively $TM\otimes L^*$; what is meant will be always clear from the context.
 Let us also remark that for two connections $\widehat\nabla$ and $\nabla$ on $TM$ related as in \eqref{cproj_change}, their induced connections on $L$ are related by
 \begin{equation}\label{cproj_change_densities}
 \widehat\nabla_a\sigma=\nabla_a\sigma-\Upsilon_a\sigma.
 \end{equation}

Moreover, we will denote by $S^2_+T^*M$ the bundle of $J$-Hermitian symmetric covariant tensors on $(M,J)$, that is, $h_{ab}\in\Gamma(S^2_+T^*M)$, if 
  \begin{equation*}
  h_{ab}=h_{(ab)}\quad \textrm{ and }\quad J_{a}{}^cJ_{b}{}^dh_{cd}=h_{ab}.
 \end{equation*} 
 Note that, by definition, any (pseudo-)K\"ahler metric $g_{ab}$ on $(M, J)$ is a non-degenerate section of $S^2_+T^*M$ and as usual we denote its inverse by $g^{ab}\in\Gamma(S^2_+TM)$, which is characterised 
 by $g_{ab} g^{bc}=\delta_{a}{}^c$ (in index-free notation we also write $g^{-1}$ for the inverse of $g$). We denote by $\vol(g)\in\Gamma(\Lambda^{2n}T^*M)$ the volume form 
 of $g$. Given a (pseudo-)K\"ahler metric $g$ on $(M,J)$, one can form the $(n+1)$st root of its volume form $\vol(g)$ (viewed as a positive section of an oriented line bundle), which naturally defines a positive section $\vol(g)^{\frac{1}{n+1}}$ of $L$. Hence, on a (pseudo-)K\"ahler manifold
 $\vol(g)$ respectively $\vol(g)^{\frac{1}{n+1}}$ canonically trivialise $\Lambda^{2n}T^*M$ respectively $L$ and hence sections of both these line bundles can be canonically identified with functions. Note also that any (pseudo-)K\"ahler metric gives rise to a c-projective structure via its Levi-Civita connection.
  
\subsection{Metrisability equation}  
Suppose $(M,J,[\nabla])$ is a c-projective manifold of real dimension $2n\geq 4$. Then the Leibniz rule together with  \eqref{cproj_change} and \eqref{cproj_change_densities}
implies that the trace-free part of $\nabla_a\eta^{bc}$ for any section $\eta^{bc}\in\Gamma(S^2_+TM\otimes L)$ is independent of the choice of connection $\nabla\in[\nabla]$. Hence, in this sense the so-called \emph{metrisability equation}, given by

\begin{equation}\label{metrisability_equation}
 \nabla_a\eta^{bc}-\delta_{a}{}^{(b}X^{c)}- J_a{}^{(b}J_d{}^{c)} X^{d}=0,
\end{equation}
 where $X^a=\frac{1}{n}\nabla_a\eta^{ab}\in\Gamma(TM\otimes L)$, 
 is c-projectively invariant. A detailed analysis of the geometry and algebra of this equation can be found in \cite{CEMN}, we recall here only the properties relevant for this article:

\begin{itemize}
\item Since \eqref{metrisability_equation} is linear, its solution space $\mathcal S$ is a real vector space. 
\item
C-projective invariance implies that the pull-back of any solution of \eqref{metrisability_equation} by a c-projective transformation is again a solution. 
\item
If $(M,J,[\nabla])$ admits a \emph{compatible} K\"ahler metric, that is $[\nabla]$ contains the Levi-Civita connection of a K\"ahler metric $g$ of $(M,J)$, then
the section 
\begin{equation}\label{eta}
\eta^{ab}:=g^{ab}\vol(g)^{\frac{1}{n+1}}:=g^{ab}\otimes\vol(g)^{\frac{1}{n+1}}\in\Gamma(S^2_+TM\otimes L) 
\end{equation}
defines a solution of
 \eqref{metrisability_equation}, since $g$ and consequently $\vol(g)^{\frac{1}{n+1}}$ are parallel for the Levi-Civita connection of $g$. In fact, by \cite[Proposition 4.5]{CEMN}\label{Metri_Prop}, mapping an Hermitian metric $g_{ab}$ (of arbitrary signature) on $(M,J)$ to $\eta^{ab}$ defined by \eqref{eta} restricts to a bijection between compatible (pseudo-)K\"ahler metrics of $(M,J, [\nabla])$ and solutions of \eqref{metrisability_equation} 
 that are non-degenerate (viewed as bundle maps $T^*M\times T^*M\rightarrow L$) at any point of $M$.
 \item Note also that, since $L$ is an oriented line bundle, we have not only a well defined notion for sections  $T^*M\times T^*M\rightarrow L$ of $S^2_+TM\otimes L$, in particular of $\mathcal S$, to be non-degenerate at a point
 but also to be positive-definite, negative-definite and indefinite at a point, since these notions are independent of the choice of a positive trivialising section of $L$. 
 \end{itemize}
 \begin{rem}
Since \eqref{metrisability_equation} is an overdetermined system of PDEs, for a generic c-projective structure, one has $\mathcal S=\{0\}$. In particular, a c-projective structure has generically no compatible (pseudo-)K\"ahler metrics.
 \end{rem}
Suppose now $(M,J,g)$ is a (pseudo-)K\"ahler manifold with Levi-Civita connection $\nabla$ and consider the induced c-projective structure $(M,J,[\nabla])$. Then we can
use $g$ to identify contravariant with covariant tensors and $\vol(g)^{\frac{1}{n+1}}$ to trivialise $L$.
 In particular, any section $\eta^{ab}\in\Gamma(S^2_+TM\otimes L)$ can be identified with a $(J,g)$-Hermitian endomorphism $A$ of $TM$
 given by 
 \begin{equation}
A_{a}{}^b:=\vol(g)^{-\frac{1}{n+1}}\,\eta^{bc}g_{ca}\in\Gamma(\textrm{End}(TM)),
 \end{equation} 
 where $\vol(g)^{-\frac{1}{n+1}}\in\Gamma(L^*)$ is the dual section of $\vol(g)^{\frac{1}{n+1}}\in\Gamma(L)$ (note that $L\otimes L^*$ is the trivial bundle $M\times \R$).
Moreover, solutions of \eqref{metrisability_equation} can be identified with $(J,g)$-Hermitian endomorphism $A\in\Gamma(\textrm{End}(TM))$ that satisfy (with respect to the Levi-Civita connection $\nabla$ of $g$)
 \begin{equation}\label{mobility_eq}
 \nabla_a A_{b}{}^c=\frac{1}{2}( g_{ab}\Lambda^c+\delta_{a}{}^c \Lambda_b+\Omega_{ab}J_{d}{}^c\Lambda^d+J_{a}{}^c\Lambda^d\Omega_{db}) \quad\textrm{ for some } \Lambda^c\in\Gamma(TM),
 \end{equation}
  where $\Omega_{ab}:=J_{a}{}^cg_{cb}$ is the K\"ahler form of $g$ and $\Lambda_a:=\Lambda^bg_{ba}$.
 Note that, if we set $\lambda:= \frac{1}{2} A_{a}{}^a$,  then $\Lambda^a$ equals the gradient $\textrm{gr}(\lambda)$ of $\lambda$ with respect to $g$. 
The equation \eqref{mobility_eq} is precisely the form of the metrisability equation
that was used in \cite{DM, FKMR, S} to study c-projectively equivalent K\"ahler metrics; we will refer to it also as the \emph{mobility equation}.
 
  In summary, $g$ induces an isomorphism of vector spaces
  \begin{equation*}
 \mathcal S\simeq \textrm{Sol}(J,g),
 \end{equation*}
 where 
 \begin{equation*}
 \textrm{Sol}(J,g):=\{A_{a}{}^b\in\Gamma(\textrm{End}(TM)): A \textrm{ is } (J,g)\textrm{-Hermitian and satisfies } \eqref{mobility_eq}\}.
\end{equation*}
 Note that invertible elements in 
$ \textrm{Sol}(J,g)$ correspond to (everywhere) non-degenerate elements in $\mathcal S$ and hence to (pseudo-)K\"ahler metrics that are c-projectively equivalent to $g$. 
Given an invertible solution $A\in \textrm{Sol}(J,g)$ the corresponding c-projectively equivalent metric $\tilde g$ is given by
\begin{equation*}
\tilde g_{ab}=\sqrt{\det{}_\R(A)}\,g_{ac} B_{b}{}^c=|\det{}_{\C}(A)|\,g_{ac} B_{b}{}^c,
\end{equation*}
where $B=A^{-1}$ and $\det_\R(A)$ and $\det_\C(A)$ denote the real and complex determinant of $A$, respectively (note that, since $A$ is $(g,J)$-Hermitian, $\det_\C(A)$ is a real-valued function).
Moreover, $\tilde g$ is evidently affinely equivalent to $g$ (i.e.\,$\nabla$ is also the Levi-Civita connection of $\tilde g$) if and only if $A$ is $\nabla$-parallel.

Since $ \textrm{Sol}(J,g)$ always contains the identity, one can at least locally add to any solution $A\in \textrm{Sol}(J,g)$ always an appropriate multiple of the identity to obtain 
an invertible element of $ \textrm{Sol}(J,g)$. Hence, locally the dimension of $ \textrm{Sol}(J,g)$ coincides with the number of linearly independent compatible (pseudo-)K\"ahler metrics of $(M,J,[\nabla])$.
Further let us remark that the equation \eqref{metrisability_equation} (respectively \eqref{mobility_eq}) is of finite type and prolongation shows that its solutions are in 
bijection to parallel sections of a linear connection on a vector bundle of rank $(n+1)^2$ (see \cite{CEMN,DM}). Hence, the vector space $\mathcal S$ (respectively $ \textrm{Sol}(J,g)$) is of dimension at most $(n+1)^2$. 
As already mentioned in the introduction we call the dimension of $\mathcal S$ (respectively $ \textrm{Sol}(J,g)$) the \emph{degree of mobility} 
of $(M, J,[\nabla])$ (respectively $(M,J,g)$).

For later purpose we also recall the following fact:

\begin{prop}\label{Killing}
Suppose $(M,J,g)$ is a (pseudo-)K\"ahler manifold of dimension $2n\geq4$. Then for any solution $A\in\emph{Sol}(J,g)$ of the mobility equation the corresponding vector field $\Lambda$ is 
holomorphic (i.e.\,its (local) flow\,preserves $J$) and $K:=J\Lambda$ is even a holomorphic Killing vector field with respect to $(J, g)$, which is equivalent to $\nabla\Lambda$ being $(J,g)$-Hermitian.
Moreover, $\nabla\Lambda$ commutes with $A$.
\end{prop}
\begin{proof}
The first statement is well known in c-projective geometry; see e.g.\,\cite[Proposition 5.6]{CEMN}, \cite[Lemma 1]{FKMR} or in the language of Hamiltonian 2-forms \cite[Proposition 3]{ACG}. A proof of the second statement can be found in \cite{DM}, \cite[Proposition 5.13]{CEMN} or also \cite[Lemma 2.2(7)]{BMR}.
\end{proof}
 
\subsection{Metric c-projective structures and integrals for the geodesic flow}\label{C-prj.-integrals}
 Recall that a smooth function $I\colon TM\to \R$ on a (pseudo-)Riemannian manifold $(M,g)$
 is called an \emph{integral of the geodesic
 flow} (or an \emph{integral}) of $g$, if for any affinely parametrised
geodesic $\gamma$ the function $s\mapsto I(\gamma'(s))$ is constant.

 Suppose now $(M,J,g)$ is a (pseudo-)K\"ahler manifold of dimension $2n\geq 4$. 
 In Proposition \ref{Killing} we have already noted that a solution $A\in\textrm{Sol}(J,g)$ 
 of the mobility equation \eqref{mobility_eq} gives rise to a holomorphic Killing vector field $K$ and hence to a linear integral for the geodesic flow of $g$.
 In fact, it is known that all coefficients of the characteristic polynomial of $A$ are generators 
 of holomorphic Killing vector fields and hence give rise to linear integrals of $g$ (see \cite[Proposition 3]{ACG},\cite[Theorem 5.11(1)]{CEMN}); 
 we will however only need Proposition \ref{Killing} in this article.
 
 An essential tool in our article will be that any solution $A\in\textrm{Sol}(J,g)$ of the mobility equation gives also
 rise to a family of quadratic integrals $I_t$ of $g$ as shown by Topalov in \cite{Top}. For $t\in\R$ the integral $I_t$ is given by 
 \begin{equation}\label{Integrals}
 I_t\colon TM\to \R, \quad I_t(X)=g( {\det{}_{\C}(A(t))}A(t)^{-1}X,X),
\end{equation} 
where $A(t):=t\textrm{Id}-A$ and $\det_{\C}(A(t))$ denotes the complex determinant of $A(t)$ (which is here viewed as a complex $n\times n$-matrix). 
Note that $\det_{\C} (A(t))A(t)^{-1}$ is a polynomial of degree $n-1$ in $t$, whose coefficients give rise to integrals of $g$.
Moreover, it was shown in \cite[Theorem 5.18]{CEMN} that on an open dense set the degree of the minimal polynomial of $A$ (i.e. in case $g$ is positive definite, the number of different eigenvalues of $A$) is constant and equals the number of functionally independent integrals in the family $I_t$.
 
 \subsection{C-projective Weyl curvature}\label{curvature}
 Suppose $(M,J)$ is a complex manifold of real dimension $2n\geq 4$. Let $\nabla$ be a complex torsion-free connection 
 $\nabla$ on $TM$ and write
  \begin{equation*}
 R_{ab}{}^{\,c}{}_dX^d:=\nabla_a\nabla_b X^c-\nabla_b\nabla_a X^c.
 \end{equation*}
 for its curvature and $\Ric_{ab}:=R_{ca}{}^{\,c}{}_b$ for its Ricci tensor. Then one may decompose $R_{ab}{}^{\,c}{}_d$ as
 
\begin{equation}\label{Weyl_curv}
R_{ab}{}^{\,c}{}_d=W_{ab}{}^{\,c}{}_d+(\partial \Rho)_{ab}{}^{\,c}{}_d,
\end{equation}
 where 
\begin{equation*}\label{partial}
(\partial\Rho)_{ab}{}^{\,c}{}_d:=
\delta_{[a}{}^{c}\Rho_{b]d}
-J_{[a}{}^c\Rho_{b]e}J_d{}^e
-\Rho_{[ab]}\delta_{d}{}^c
- J_{[a}{}^e\Rho_{b]e}J_d{}^{c}.
\end{equation*}

\begin{equation}\label{RhoTensor}
\Rho_{ab}:=\tfrac{1}{n+1}(\Ric_{ab}+
\tfrac{1}{n-1}(\Ric_{(ab)}-J_{(a}{}^c J_{b)}{}^d
\Ric_{cd})).
\end{equation}

It can be shown (see \cite{OT, CEMN}) that $W_{ab}{}^{\,c}{}_d$ does not depend on the connection in the induced c-projective class $[\nabla]$ of $\nabla$.
Hence, it is an invariant of the c-projective manifold $(M,J,[\nabla])$, called its \emph{c-projective Weyl curvature}.

Recall also that a K\"ahler metric $g$ on $(M,J)$ is said to have \emph{constant holomorphic sectional curvature} $\mu\in\R$, if its curvature (i.e.\,the curvature of its Levi-Civita connection) takes the form
\begin{equation*}
R_{abcd}=\frac{\mu}{4} (g_{ac}g_{bd}-g_{bc}g_{ad}
+\Omega_{ac}\Omega_{bd}
-\Omega_{bc}\Omega_{ad}
+2\Omega_{ab}\Omega_{cd}),
\end{equation*}
where $R_{abcd}:=R_{ab}{}^{\,e}{}_d\,g_{ec}$ and $\Omega_{ab}:=J_a{}^cg_{cb}\in\Gamma(\Lambda^2T^*M)$ is the K\"ahler-form. 

In the following theorem we collect some results which we will need in the sequel: 

\begin{thm} \label{Weyl_fund_invariant}
Suppose $(M,J,g)$ is a K\"ahler manifold of dimension $2n\geq 4$ with Levi-Civita connection $\nabla$. Then one has:
\begin{enumerate}
\item  $W$ vanishes identically if and only if $(M, J, [\nabla])$ is locally c-projectively flat, that is, locally c-projectively equivalent to $(\CP^n, J,[\nabla^{g _{FS}}])$, where $\nabla^{g _{FS}}$ denotes the Levi-Civita connection of the Fubini--Study metric $g_{FS}$.
\item  If $M$ is connected, then $W$ vanishes identically if and only if $(M,J,g)$ has constant holomorphic sectional curvature.
\item  If $(M,J,g)$ is connected, complete and has positive constant holomorphic sectional curvature, then $(M,J,g)$ is simply-connected and isometric to $(\CP^n, J, cg _{FS})$ for some positive constant $c$.
\end{enumerate}
\end{thm}
\begin{proof}
Statement $(3)$ is a standard fact in K\"ahler geometry and for $(1)$ and $(2)$ see \cite{T} or \cite[Theorems 2.16 and 4.2]{CEMN}.
\end{proof}

\begin{rem}
For a general c-projective manifold (which is not necessarily induced by a K\"ahler metric) statement $(1)$ of Theorem \ref{Weyl_fund_invariant} still holds provided $2n\geq6$.
If $2n=4$, the c-projective Weyl curvature is in general not sufficient to characterise c-projective flatness.
It turns out that in this case the vanishing of $W$ and a part of the c-projective Cotton-York tensor is what characterises c-projectively flat structures, see \cite[Theorem 2.16]{CEMN}. 
\end{rem}

\begin{rem}\label{deg_on_CP}
We already mentioned that on a K\"ahler manifold $(M,J,g)$ of real dimension $2n\geq 4$ solutions of the mobility equation are in bijection to parallel sections of a linear connection on a vector bundle of rank $(n+1)^2$, see e.g.\,\cite[Theorem.\,4.16]{CEMN}. Hence, if $M$ is simply-connected, $\dim(\textrm{Sol}(J,g))=(n+1)^2$ if and only if this connection has vanishing curvature, which in turn can be shown to be the case if and only if the c-projective Weyl curvature vanishes. In particular, in view of Theorem \ref{Weyl_fund_invariant}, on $(\CP^n, J, g _{FS})$ we have $\dim((\textrm{Sol}(J,g_{FS}))=(n+1)^2$.
\end{rem}

\section{K\"ahler manifolds with a very special type of solution of the mobility equation}\label{special_sol}
In this section we study the topology of complete K\"ahler manifolds admitting a solution of the mobility equation 
of a very restrictive type. We show that the existence of such a solution implies that  the manifold is compact. 
This will be a crucial ingredient in the proof of our main Theorem \ref{main_thm}. 

\subsection{Some general facts}
Suppose $(M,J,g)$ is a connected K\"ahler manifold and $A\in\textrm{Sol}(J,g)$ a solution of the mobility equation. Since $g$ is assumed to be positive-definite and $A$ is $(J,g)$-Hermitian, at any point of the manifold 
$A$ is diagonalisable with real eigenvalues and any eigenvalue is of even (real) algebraic multiplicity.

\begin{defin}\label{reg_points}
A point $x\in M$ is called \emph{regular} with respect to $A$, if
\begin{itemize}
\item the number of distinct eigenvalues of $A$ is constant on a neighbourhood of $x$, 
\item for a smooth eigenvalue $\rho$ defined on a neighbourhood of $x$ either $d\rho(x)\neq 0$ or $\rho$ is constant on a neighbourhood of $x$.
\end{itemize}
\end{defin}
We denote the set of regular points by $M_{\textrm{reg}}$. Note that $M_{\textrm{reg}}$ is an open and dense subset of $M$. 
Moreover, the following can be shown (see \cite[Lemma 5.16, Corollary 5.17]{CEMN} or also \cite[Lemma 2.2]{BMR}):
\begin{prop}\label{prop_eigenvalues} 
Suppose $(M,J,g)$ is a connected K\"ahler manifold of dimension $2n\geq 4$ and $A\in\emph{Sol}(J,g)$ a solution of the mobility equation. Then we have:
\begin{enumerate}
\item At any regular point the algebraic (real) multiplicity of any non-constant eigenvalue $\rho$ of $A$ is $2$ and its eigenspace is spanned by its gradient $\emph{gr}(\rho)$ and its skew-gradient 
$J\emph{gr}(\rho)$.
\item If an eigenvalue $\rho$ of $A$ is constant around some regular point $x$ (i.e. $d\rho(x)=0$), then the constant $\rho$ is an eigenvalue at any point of $M$ and its (real) algebraic multiplicity is constant on the set of regular points. \end{enumerate}
\end{prop}

\subsection{Solutions of the mobility equation of special type}
 
Let $(M,J,g)$ be a connected K\"ahler manifold of dimension $2n\geq 4$ and $A\in\textrm{Sol}(J,g)$ a solution of the mobility equation with the following property:
\begin{enumerate}
\item[{(P)}] it  
has locally around any regular point the following structure of eigenvalues: 
\begin{itemize}
\item two constant eigenvalues $1$ and $0$ of multiplicity $2m$ and $2\tilde m$ respectively,
\item one non-constant eigenvalue $\rho$ with values in $(0,1)$ of multiplicity $2$,
\end{itemize}
where $m,\tilde m\in \Z_{\geq0}$ are arbitrary such that $n-m-\tilde m=1$.
\end{enumerate}

Note that assumption $\textrm{(P)}$ implies that $\rho=\lambda-m$, where $\lambda=\frac{1}{2} A_{a}{}^a$. Since $\lambda$ defines a smooth function on all of $M$, we can also   extend $\rho$ from a smooth function defined on the set of regular points to a smooth function defined on the whole manifold $M$ by dint of this equality. We set 

$$M_0:=\{x\in M: \rho(x)=0\}\,\, \textrm{ and }\,\, M_1:=\{x\in M: \rho(x)=1\}.$$
Since $\rho$ has extrema at points of $M_0$ and $M_1$, we have

$$M_0\subseteq N\quad\quad\textrm{ and }\quad\quad M_1\subseteq N,$$
where $N$ is the zero set of the gradient vector field $\Lambda:=\textrm{gr}(\lambda)=\textrm{gr}(\rho)\neq0$. Since $N$ also coincides with the zero set of $J\Lambda$, which is a (holomorphic) Killing vector field by Proposition \ref{Killing}, 
$N$ is a union of closed connected totally geodesic submanifolds, each of which has even dimension at most $2n-2$.

\begin{rem}
To avoid any ambiguity, let us remark that \emph{closed submanifold} of $M$ here and everywhere else in this article, in particular in Propositions \ref{M0M1} and \ref{foliation} below, means that it is a closed subset of $M$, not that it is compact.
\end{rem}

For later purposes, note that for any point $x\in M$ we can find a basis of $T_xM$, which we call \emph{adapted},
in which $g_x=\textrm{Id}=\textrm{diag}(1,...,1)$ and $A_x$ and $J_x$ have the following block-diagonal form:

\begin{equation}\label{diagonal-1}
A_x= \left(\begin{array}{ccc}
 \rho(x) \textrm{Id}_{2}&&\\
& \textrm{Id}_{2m}&\\
&& 0_{2\tilde m}
\end{array}\right) \quad \textrm{and} \quad J_x=  \left(\begin{array}{ccc} \begin{matrix}0  &-1 \\ 1 & 0\end{matrix}&&\\&\ddots&\\&& \begin{matrix}0  &-1 \\ 1 & 0\end{matrix}\end{array}\right).
\end{equation}

For a vector $\xi\in T_xM$ 
we shall denote by $\xi_i$ the coordinates 
with respect to a chosen adapted basis of $T_xM$.  
In such coordinates, the family of quadratic integral $I_t$ \eqref{Integrals}  induced by $A$ has the form:
\begin{equation}\label{INT}
 I_t( \xi)= (t-1)^{\alpha(m)(m-1)}t^{\alpha(\tilde m)(\tilde m-1)} \tilde I_t( \xi),
  \end{equation}
where  
\begin{equation*}
 \tilde I_t( \xi)= (t-1)^{\alpha(m)}t^{\alpha(\tilde m)}(\xi_1^2 + \xi_2^2)+(t-\rho)t^{\alpha(\tilde m)}(\xi_3^2+...+\xi_{2m+2}^2)+(t-\rho)(t-1)^{\alpha(m)}(\xi_{2m+3}^2+...+\xi_{2n}^2),
 \end{equation*} 
and $\alpha(\ell)$ equals $1$ for $\ell\geq 1$ and $0$ for $\ell =0$.
Note that for fixed $t$ the coefficient  
$(t-1)^{\alpha(m)(m-1)}t^{\alpha(\tilde m)(\tilde m-1)}$ is a constant and hence 
$\tilde I_t( \xi)$ also forms a family of integrals for the geodesic flow of $g$. 
\\\\
The goal of this section is to prove:

\begin{thm} \label{M_compact}
Suppose $(M, J, g)$ is a connected complete K\"ahler manifold of dimension $2n\geq 4$ and let $A\in\emph{Sol}(J,g)$ be a solution of the mobility equation that satisfies $\emph{(P)}$. If $M_0$ and $M_1$ are both not empty, then $M$ is compact.
\end{thm}

The proof will be based on several propositions:

\begin{prop} \label{M0M1}
Suppose $(M, J, g)$ is a connected K\"ahler manifold of dimension $2n\geq 4$ and $A\in \emph{Sol}(J,g)$ satisfies property $\emph{(P)}$. 
Then one of the following statements holds for $M_0$ (respectively $M_1$):
\begin{itemize}
\item $M_0$ (resp.\,$M_1$) is empty, or if not,
\item $M_0$ (resp.\,$M_1$) is a discrete subset of $M$ provided $m=0$ (resp.\,$\tilde m=0$) and
a closed totally geodesic K\"ahler submanifold of dimension $2m$  (resp.\,$2\tilde m$) provided $m\geq 1$ (resp.\,$\tilde m\geq 1$) whose tangent space at any point is the eigenspace of $A$ 
with eigenvalue $1$ (resp.\,$0$). In particular, if $M$ is complete, then $M_0$ (resp.\,$M_1$) is complete.
\end{itemize}
\end{prop}
\begin{proof}
It suffices to prove the statement for $M_0$, since replacing $A$ by $\textrm{Id}-A$ (which obviously also satisfies $\textrm{(P)}$) interchanges $M_0$ and $M_1$.

Suppose $M_0\neq \emptyset$. 
Fix a point $x_0\in M_0$ and a convex neighbourhood $U$ of $x_0$ and let $\varepsilon >0$ be sufficiently 
small such that the image of
$$S_{x_0}^{\varepsilon}:=\{\xi\in T_{x_0}M: \xi_i=0 \textrm{ for } i=1,2 \textrm{ or } i\geq 2m+3, \textrm{ and } \xi_{3}^2+...+\xi_{2m+2}^2<\varepsilon\}$$
under the exponential map is contained in $U$.
We aim to show that, by possibly shrinking $U$, we can achieve that $\exp(S_{x_0}^{\varepsilon})=U\cap M_0$. Note that, since $x_0\in M_0$ was
arbitrary, this would imply that $M_0$ is a totally geodesic closed K\"ahler submanifold of dimension $2m$ (or a discrete subset, if $m=0$) as desired. 
 
Now consider the family $\tilde I_t$ of integrals defined as in \eqref{INT}
and let $\gamma: [0,1]\rightarrow U$ be a geodesic with $\gamma(0)=x_0$ and 
$\gamma'(0)=(\xi_1,..., \xi_{2n})\in S_{x_0}^{\varepsilon}\setminus\{0\}$.  
Then,   $$\tilde I_t(\gamma'(0))=t^{\alpha(\tilde m)+1}(\xi_{3}^2+...+\xi_{2m+2}^2),$$ which has a zero of order $\alpha(\tilde m)+1$ at $t=0$. Hence the 
same must be true at $\gamma(1)$ and substituting $t=0$ in \eqref{INT}
we therefore obtain
\begin{equation}
\tilde I_0(\gamma'(1))= (-1)^{\alpha(m)}0^{\alpha(\tilde m)}(\xi_1^2 + \xi_2^2)-0^{\alpha(\tilde m)}\rho(\xi_3^2+...+\xi_{2m+2}^2) +(-1)^{\alpha(m)+1}\rho(\xi_{2m+3}^2+...+\xi_{2n}^2) =0, \label{ordering}
\end{equation}
where the $\xi_i$'s now denote the coefficients of $\gamma'(1)$ with respect to an adapted basis of $T_{\gamma(1)}M$.
\textbf{Case 1:} Suppose $\tilde m=0$.
\\Then have $$\tilde I_0(\gamma'(1))=-\left((\xi_1^2 + \xi_2^2)+\rho(\xi_3^2+...+\xi_{2n}^2)\right)=0,$$ which implies $\rho(\gamma(1))=0$, since $\gamma'(1)\neq 0$. 
Hence, $\exp(S_{x_0}^{\varepsilon})\subset U\cap M_0$ in this case. Moreover, $M_0$ is contained in the zero set of the non-trivial Killing vector field $J\Lambda$, 
which equals a union of connected submanifolds of dimension at most $2n-2$. For dimensional reasons we can hence achieve by possibly shrinking $U$ that $\exp(S_{x_0}^{\varepsilon})=U\cap M_0$.

\noindent
\textbf{Case 2:} Suppose $\tilde m\geq 1$.
\\If $\rho(\gamma(1))\neq 0$, then \eqref{ordering} implies that at $\gamma(1)$ we have $\xi_i=0$ for $2m+3\leq i\leq 2n$ and hence
$$ \tilde I_t(\gamma'(1))=  \left(t\left((t-1)^{\alpha(m)}(\xi_1^2 + \xi_2^2)+(t-\rho)(\xi_3^2+...+\xi_{2m+2}^2)\right)\right).$$ Since $\tilde I_t(\gamma'(1))$ must have a zero of order $\alpha(\tilde m)+1=2$ at $t=0$, 
this implies that also $\xi_i=0$ at $\gamma(1)$ for $1\leq i\leq 2m+2$, which contradicts our assumptions. 
Thus, we must have $\rho(\gamma(1))=0$ and hence $\exp(S_{x_0}^{\varepsilon})\subset U\cap M_0$.

Now take a point $y\in U\setminus (U\cap M_0)$ (since the set of regular point is dense in $M$, almost every point in $U$ has this property) and let $x\in M_0\cap U$. 
Then there exists a geodesic $\gamma: [0,1]\rightarrow U$ connecting $y=\gamma(0)$ with $x=\gamma(1)$. Since $\rho$ is zero at $x$, we see that the 
integral $\tilde I_0$ is zero along $\gamma$ and hence we deduce from \eqref{INT} that $\xi_{2m+3}=...=\xi_{2n}=0$ at $y$. Thus, $M_0\cap U\subset \exp(V_y)$, 
where $V_y$ is the $2m+2$-dimensional subspace of $T_yM$  defined by the condition  $\xi_{2m+3}=...=\xi_{2n}=0$. Since $y$ was arbitrary in $U\setminus U\cap M_0$, 
we must have 
\begin{equation}\label{intersection}
M_0\cap U\subset \bigcap_{y\in  U\setminus U\cap M_0}\exp(V_y).
\end{equation}
Since $\exp(V_y)$ is locally a $2m+2$-dimensional submanifold and $2m+2<2n$, the intersection  $\bigcap_{y\in  U\setminus U\cap M_0}\exp(V_y)$ 
must locally be a submanifold of dimension at most $2m$. Since we moreover have $\exp(S_{x_0}^{\varepsilon})\subset U\cap M_0$, we conclude again that we can arrange
$\exp(S_{x_0}^{\varepsilon})=M_0\cap U$ by possible shrinking $U$. 
 \end{proof}
 
Using the quadratic integrals \eqref{INT}, we next establish some key properties of certain types of geodesics of K\"ahler manifolds with property (P):

\begin{prop}\label{geodesics1} 
Suppose $(M, J, g)$ is a connected K\"ahler manifold of dimension $2n\geq 4$ and $A\in\emph{Sol}(J,g)$ satisfies property $\emph{(P)}$. Then the following holds:
\begin{enumerate}
\item Let $\gamma:I\rightarrow M$ be a geodesic and suppose there exists $s_0\in I$ such that $\Lambda(\gamma(s_0))\neq 0$ and $\gamma'(s_0)$
is proportional to $\Lambda(\gamma(s_0))$. Then for all $s\in I$ such that $\Lambda(\gamma(s))\neq 0$ the velocity vector $\gamma'(s)$ is proportional to $\Lambda(\gamma(s))$.
\item The distribution $D$ of rank $2m+2$ on $M\setminus M_0$ generated by the eigenvectors corresponding to the eigenvalues $\rho$ and $1$ is totally geodesic.
\item If $M_1\neq \emptyset$, then for any geodesic $\gamma: I\rightarrow M\setminus M_0$ with $\gamma(s_0)\in M_1$ and $\gamma'(s_0)\in T^\perp_{\gamma(s_0)}M_1$ for some $s_0\in I$, the velocity vector $\gamma'(s)$
is a $\rho$-eigenvector of $A$ for all $s\in I$. Moreover, it is proportional to $\Lambda(\gamma(s))$ provided the latter is not zero.
\item If $M_0\neq\emptyset$ and $M_1\neq\emptyset$, then any geodesic connecting a point of $M_0$ with a point of $M_1$ is orthogonal to $M_0$ respectively $M_1$ at these points.

\end{enumerate}

\end{prop}
\begin{proof}
(1) With respect to an adapted basis of $T_{\gamma(s_0)}M$, the family of integrals $\tilde I_t$ defined as in \eqref{INT} satisfies (we again assume $\gamma'(s_0)=(\xi_1,...,\xi_{2n})$) that
$$\tilde I_t(\gamma'(s_0))=(t-1)^{\alpha(m)}t^{\alpha(\tilde m)}(\xi_1^2+\xi_2^2),$$ by 
Proposition \ref{prop_eigenvalues}(1). 
Hence, $\tilde I_0(\gamma'(s_0))=0$ for $\tilde m\geq 1$ and $\tilde I_1(\gamma'(s_0))=0$ for $m\geq 1$. 
Therefore at any point $\gamma(s)$ of the geodesic we must have
\begin{align} \label{EqInt1}
\tilde I_0(\gamma'(s))=(-1)^{\alpha(m)+1}\rho(\xi^2_{2m+3}+...+\xi^2_{2n})=0\quad \textrm{ if } \tilde m\geq 1,\\
\tilde I_1(\gamma'(s))=(1-\rho)(\xi^2_{3}+...+\xi^2_{2m+2})=0\quad  \textrm{ if } m\geq 1,  \label{EqInt2}
\end{align}
where $\{\xi_i\}$ now denotes the coordinates of $\gamma'(s)$ with respect to an adapted basis of $T_{\gamma(s)}M$.
At points where $\rho$ equals $0$ or $1$, the function $\rho$ has an extremum and so $\Lambda=\textrm{gr}(\rho)=0$ at these points. Hence, we conclude from \eqref{EqInt1} and \eqref{EqInt2} 
that $\gamma'(s)$ is a linear combination of $\Lambda$ and $J\Lambda$ at all points $\gamma(s)$ where these vector fields are not zero. Since $J\Lambda$ is a Killing vector field and orthogonal 
to $\gamma'(s_0)$, the same must be true for $\gamma'(s)$ for any $s$, which shows that $\gamma'(s)$ is proportional to $\Lambda$ at all points $\gamma(s)$ where the latter does not vanish.

$(2)$ Note first that for $\tilde m=0$ the statement is trivially satisfied, since $D=TM|_{M\setminus M_0}$ in this case. Assume now $\tilde m\geq 1$ and 
consider a geodesic $\gamma: I\rightarrow M\setminus M_0$ such that $\gamma'(s_0)\in D_{\gamma(s_0)}$ for some $s_0\in I$. Then $\tilde I_0(\gamma'(s_0))=0$ and hence for all $s\in I$ we have
$$\tilde I_0(\gamma'(s))=(-1)^{\alpha(m)+1}\rho(\xi_{2m+3}^2+...+\xi_{2n}^2)=0$$ with respect to an adapted basis of $T_{\gamma(s)}M$. Since $\rho(\gamma(s))\neq 0$, this implies 
$\gamma'(s)\in D_{\gamma(s)}$ for all $s\in I$.

$(3)$ Consider a geodesic $\gamma: I\rightarrow M\setminus M_0$ with $\gamma(s_0)\in M_1$ and $\gamma'(s_0)\in T_{\gamma(s_0)}^{\perp}M_1$. Note that, by Proposition \ref{M0M1},  $T_{\gamma(s_0)}^{\perp}M_1=D_{\gamma(s_0)}$. In particular, if $m=0$, the first part of the statement holds by (2). Now suppose $m\geq 1$. Then $\tilde I_1(\gamma'(s_0))=0$ and hence $\tilde I_1(\gamma'(s))=0$ for all $s\in I$. By (2), this implies that $\gamma'(s)$ is a $\rho$-eigenvector at all points with $\rho(\gamma(s))\neq 1$ and hence at all points of $\gamma(s)$ by continuity.
For the second statement, recall that, by Proposition \ref{prop_eigenvalues}, $\gamma'(s)$ is a linear combination of $\Lambda$ and $J\Lambda$ at all point $\gamma(s)$ where these vector fields do not vanish.
Moreover, since $\rho$ has a maximum at $\gamma(s_0)$, the Killing vector field $J\Lambda$ vanishes at $\gamma(s_0)$ and hence the inner product of $J\Lambda$ with $\gamma'(s)$ must be zero for all $s\in I$. 

$(4)$ Consider a geodesic $\gamma:[0,1]\rightarrow M$ such that $\gamma(0)\in M_0$ and $\gamma(1)\in M_1$. Since $\rho(\gamma(0))=0$, it follows from \eqref{INT} that 
$\tilde I_0(\gamma'(0))=0$. Hence, also $\tilde I_0(\gamma'(1))=0$ and, since $\rho(\gamma(1))=1$, we conclude from \eqref{INT} and Proposition \ref{M0M1} 
that $\gamma'(1)\in T^{\perp}_{\gamma(1)}M_1$. Similarly, since $\rho(\gamma(1))=1$, formula \eqref{INT} shows $\tilde I_{1}(\gamma'(1))=0$. Hence, also $\tilde I_{0}(\gamma'(0))=0$. Since $\rho(\gamma(0))=0$, this implies, by
 \eqref{INT} and Proposition \ref{M0M1}, that $\gamma'(0)\in T^{\perp}_{\gamma(0)}M_0$.
\end{proof}

\begin{rem}
Since the roles of $M_0$ and $M_1$ exchange, if one replaces $A$ by $\textrm{Id}-A$, the analogues of the statements $(2)$ and $(3)$ in Proposition \ref{geodesics1} also hold for $M_1$. 
\end{rem}

For a solution $A\in\textrm{Sol}(J,g)$ of the mobility equation with property (P), the set of regular points $M_{\textrm{reg}}\subseteq M$ as defined in Definition \ref{reg_points} simply coincides with the set of points of $M$ on which $\Lambda$ does not vanish, equivalently, on which the holomorphic Killing vector field $J\Lambda$ does not vanish. We have already mentioned that $M_{\textrm{reg}}$ is an open dense subset in $M$. Since $M_{\textrm{reg}}$
is the complement of the zero set of a non-trivial Killing vector field, it is the complement of a submanifold of co-dimension at least two and as such it is connected.
 To study the topology of $M$ let us consider 
the foliation $\mathcal F$ on $M_{\textrm{reg}}$ generated by the function $\rho$, that is, its leaves are the connected components of the level sets of $\rho$. Since $\textrm{gr}(\rho)=\Lambda$ does not vanish on $M_{\textrm{reg}}$, they are connected $2n-1$-dimensional submanifolds of $M$. We write $L(x)\subset M_{\textrm{reg}}$ for the leaf containing $x\in M_{\textrm{reg}}$.

\begin{lem}\label{norm_lambda} 
Suppose $(M,J,g)$ is a connected K\"ahler manifold of dimension $2n\geq 4$ and $A\in\emph{Sol}(J,g)$  has property $\emph{(P)}$. Let $\xi\in T_xM$ for some $x\in M$ be a linear combination of $K=J\Lambda$ and 
of eigenvectors corresponding to eigenvalues $0$ and $1$. Then the derivative of the function $g(\Lambda,\Lambda)=g(K,K)$ vanishes in direction of $\xi$.
\end{lem}
\begin{proof}
At a zero $x\in M$ of $\Lambda$ the function $g(\Lambda,\Lambda)$ has a minimum and hence its derivative vanishes at $x$. Suppose now $x\in M$ is regular, that is, $\Lambda(x)\neq0$ (in particular $\rho(x)\neq 0,1$). 
Recall that by Proposition \ref{Killing} the endomorphism $\nabla \Lambda$ of $TM$ commutes with $A$ and $J$, and is self-adjoint with respect to $g$. Hence,
it follows from Proposition \ref{prop_eigenvalues} that $\nabla_{\Lambda}\Lambda$ is a $\rho$-eigenvector of $A$ and, since $\nabla \Lambda$ is $(J,g)$-Hermitian, moreover that it is proportional to $\Lambda$ at $x$.
Self-adjointness of $\nabla \Lambda$ with respect to $g$ (i.e.\,$g_{ac}\nabla_b\Lambda^c$ is symmetric in $a$ and $b$) implies then that $g(\nabla_\xi\Lambda, \Lambda)=g(\nabla_\Lambda\Lambda, \xi)=0$ for $\xi\in\Lambda^{\perp}$. 

\end{proof}

\begin{prop}\label{foliation} 
Suppose $(M,J,g)$ is a connected K\"ahler manifold of dimension $2n\geq 4$ and $A\in\emph{Sol}(J,g)$  has  property $\emph{(P)}$. Then the foliation $\mathcal F$ on the set of regular points $M_{\emph{reg}}$ has the following properties:
\begin{enumerate}
\item For any $x\in M_{\emph{reg}}$ the tangent space of the leaf $L(x)$ at $x$ is generated by $J\Lambda$ and the eigenvectors of $A$ with eigenvalues $0$ and $1$.
\item The function $g(\Lambda, \Lambda)$ is constant on the leaves of $\mathcal F$.
\item Each leaf is a closed subset of $M$.
\item For any $x\in M_{\emph{reg}}$ there exist an $\varepsilon >0$ and a neighbourhood $U(x)$ in $M_{\emph{reg}}$ such that 
\begin{itemize}
\item for any $t\in(-\varepsilon, \varepsilon)$  the flow $\phi_t$ of $\Lambda$ is defined on $U(x)$ and $$\phi_t(U(x)\cap L(x))\subseteq L(\phi_t(x)),$$
\item for any fixed $t\in(-\varepsilon, \varepsilon)$ the distance function $y\mapsto\textrm{dist}(y, \phi_{t}(y))$ is constant on $U(x)\cap L(x)$.
\end{itemize}
\end{enumerate}
 
\end{prop}
\begin{proof} 
\,
(1)
For any $x\in M_{\textrm{reg}}$ the tangent space $T_xL(x)$ is given by the kernel of $\nabla\rho$ at $x$, which 
equals the orthogonal complement of the gradient $\Lambda$ of $\rho$ at $x$. Hence, the statement follows 
from Proposition \ref{prop_eigenvalues}.

(2) Since $L(x)$ is by definition connected, the statement follows from (1) and Lemma \ref{norm_lambda}.

(3) Consider a sequence $(x_k)_{k\in\N}\in L(x)$ converging to a point $\tilde x\in M$. Clearly, we must have $\rho(\tilde x)=\rho(x)$, so it remains to check that $\tilde x$ is a regular point, i.e. $d\rho(\tilde x)\ne 0$. This however immediately follows from (2).

(4) Let $x\in M_{\textrm{reg}}$. Then there exists $\varepsilon>0$ and a neighbourhood $U(x)$ of $x$ inside 
$M_{\textrm{reg}}$ such that $\phi_t$ is defined on $U(x)$ for all $t\in(-\varepsilon, \varepsilon)$. We may also arrange that $U(x)\cap L(x)$ is connected. 
Since $\phi_t$ preserves $\Lambda$ and 
$U(x)\cap L(x)$ is connected, $\phi_t(U(x)\cap L(x))$ is a connected subset contained in $M_{\textrm{reg}}$ and  $\phi_t(U(x)\cap L(x))\subseteq L(\phi_t(x))$
is equivalent to $\rho\circ \phi_t$ being constant on $U(x)\cap L(x)$ for any fixed $t$.  This in turn is equivalent, over $U(x)\cap L(x)$, to 
$[\Lambda, \xi]\in\Lambda^{\perp}$ for any section $\xi\in\Lambda^{\perp}$. 
Since $\nabla_{\xi}\Lambda\in\Lambda^{\perp}$ for $\xi\in\Lambda^{\perp}$ by Lemma \ref{norm_lambda} and $\nabla$ is torsion-free, the latter is equivalent
to $\nabla_{\Lambda}\xi\in\Lambda^{\perp}$ for $\xi\in\Lambda^{\perp}$, which follows from differentiating $0=g(\Lambda, \xi)$ for 
a section $\xi\in\Lambda^{\perp}$ in direction $\Lambda$ and the fact that  $\nabla_{\Lambda}\Lambda$ is proportional to $\Lambda$. 

For the second property in (4), note that, by possibly shrinking $U(x)$ and $\varepsilon$, we can achieve that $U(x)$ is strictly convex 
and that
$$U(x)=\{\phi_t(y): y\in U(x)\cap L(x),\,\, t\in (-\varepsilon, \varepsilon)\},$$
which gives rise to flow-box coordinates on $U(x)$.

Since $\nabla_\Lambda\Lambda$ is proportional to $\Lambda$, for any $y\in U(x)\cap L(x)$ and $t_0\in(-\varepsilon, \varepsilon)$ the curve $\phi_t(y):[0, t_0]\rightarrow U(x)$ (respectively $[t_0, 0]$ if $t_0<0$) 
can be reparametrised to a geodesic segment and by strict convexity of $U(x)$ this geodesic segment is length minimising. The length of the velocity vector of the curve $\phi_t(y)$  is given by $\sqrt{g(\Lambda, \Lambda)}$ and by (2) it is a  function of $\rho$. Thus, the length of the curve $\phi_t(y)$ is   the same for all points  $y\in U(x)\cap L(x)$ and  for any fixed $t_0\in(-\varepsilon, \varepsilon)$ 
the function $y\mapsto\textrm{dist}(y, \phi_{t_0}(y))$ is constant on $U(x)\cap L(x)$.

\end{proof}

\begin{rem}  \label{vvv}  The statements (2) and (4) of Proposition \ref{foliation}  imply the existence of a coordinate system $y_1,...,y_{2n}$ in a neighbourhood of every regular point such that 
 $\rho=\rho(y_1)$ and  
 \begin{equation} g= dy_1^2 + \sum_{i,j=2}^{2n} h_{ij} dy_idy_j. \end{equation}  In this coordinate system  $\Lambda=(\rho'(y), 0,...,0)$ and its integral curves are ``vertical'' geodesics. 
\end{rem}

 \begin{prop} \label{prop:llast} Suppose $(M, J, g)$ is a connected complete K\"ahler manifold and $A\in \emph{Sol}(J,g)$ satisfies $\emph{(P)}$. 
  Let $\gamma:I\to M$ be an arc length parameterised geodesic with the following properties:
  \begin{itemize} 
  \item[(a)] the domain $I$ of $\gamma$ is either a closed interval $[a,b]$ ($a,b\in \mathbb{R}$), a ray $[a, +\infty) $ or $(-\infty, b]$, or all of $\mathbb{R}$,
  \item[(b)] $\Lambda(\gamma(s)):=\emph{gr}_{\gamma(s)}(\rho)\neq 0$ for all $s$ in the interior $I^o$ of $I$
  and $\Lambda(\gamma(s))=0$ for all $s\in I\setminus I^o$,
   \item[(c)] at some $s_0\in I^{0}$, hence by Proposition \ref{geodesics1} at all $s\in I^{o}$, 
    the velocity vector of $\gamma$ is proportional to $\Lambda$.
  \end{itemize} 
  Then the following statements hold: 
  \begin{enumerate} 
  \item  \begin{equation}\label{reg_points_foliation} 
  M_{\emph{reg}}=\bigcup_{s\in I^o} L(\gamma(s)),
  \end{equation}
  \item 
    the flow $\phi_t$ of $\Lambda$ acts simply transitively on the set of leaves,
  \item  there exist a connected ${2n-1}$ dimensional manifold $L$ and a diffeomorphism 
 $$\Psi: M_{\emph{reg}}\rightarrow L\times I^{o}$$
 with the following properties:
  
\begin{itemize} 
 \item the images of the leaves of the foliation under $\Psi$
   are the sets of the form  $L \times \{s\}$ for $s\in I^o$; in particular, the push-forward of the function $\rho$ depends only on $s\in I^o$, 
    \item the push-forward of the vector field $\Lambda$ is tangent to the lines $\{x\} \times I^o$, and  
    \item for  any  point  $x \in L$ and $s_1\le s_2 \in  I^{o}$   the distance between the points $\Psi^{-1} (x, s_1)$ and   $\Psi^{-1}(x, s_2)$ is  $s_2 -s_1$ and the 
   shortest arc-length parameterised geodesic connecting these points is the preimage $\Psi^{-1}(\tilde \gamma)$  
   of the ``vertical'' curve 
  $\tilde \gamma:[s_1, s_2  ] \to  L \times I^o\ , \, s\mapsto (x, s).$ \end{itemize} \end{enumerate}
\end{prop}
\begin{proof}
(1) Denote by $M'_{\textrm{reg}}\subseteq M_{\textrm{reg}}$ the right-hand side of \eqref{reg_points_foliation}. Since $M_{\textrm{reg}}$ is connected, it is sufficient to 
show that $M'_{\textrm{reg}}$ is open and closed in $M_{\textrm{reg}}$ in order to show that $M'_{\textrm{reg}}=M_{\textrm{reg}}$.

First, let us show that $M'_{\textrm{reg}}$ is open: Fix a point $x\in L(\gamma(s))$ with $s\in I^o$. Since $L(x)=L(\gamma(s))$ is connected, there exists a curve $\tilde c: [0,\frac{1}{2}]\rightarrow L(\gamma(s))$ connecting $\tilde c(0)=\gamma(s)$ with $\tilde c(\frac{1}{2})=x$. By (4) of Proposition \ref{foliation} and compactness of $[0, \frac{1}{2}]$, there exists $\varepsilon>0$ such that $\tilde c([0,\frac{1}{2}])$ can be covered by neighbourhoods (in $M_{\textrm{reg}}$) of the form
$$U(z)=\{\phi_t(z): y\in U(z)\cap L(z),\,\, t\in (-\varepsilon, \varepsilon)\}, \quad z\in\tilde c([0,\frac{1}{2}]),$$ satisfying the properties of (4) of Proposition \ref{foliation}. 
Let now $\phi_{t_0}(y)$ be any point in $U(x)$ (i.e. $|t_0|<\varepsilon$ and $y\in L(x)\cap U(x)$). Since $L(y)=L(x)=L(\gamma(s))$ is connected, we can extend $\tilde c$ to a curve $c: [0,1]\rightarrow L(\gamma(s))$ connecting $c(0)=\gamma(s)$ with $c(1)=y$. By construction, 
the curve $\phi_{t_0}\circ c$ connecting $\phi_{t_0}(\gamma(s))$ with $\phi_{t_0}(y)$ is well-defined and lies inside one leaf.
By (c), $\phi_{t_0}(\gamma(s))$ must lie on $\gamma(I^o)$, which shows that $\phi_{t_0}(y)\in M'_{\textrm{reg}}$. Hence, $M'_{\textrm{reg}}$ is open in $M_{\textrm{reg}}$. 

In order to show that $M'_{\textrm{reg}}$ is also closed in 
$M_{\textrm{reg}}$ suppose $(x_k)_{k\in\mathbb N}$ is a sequence in $M'_{\textrm{reg}}$ that converges to a point $\tilde x\in M_{\textrm{reg}}$. Denote by
$(s_k)_{k\in\mathbb N}\in I^o$ the corresponding sequence in $I^o$ such that $x_k \in L(\gamma(s_k))$. 
 Since $g(\Lambda(\gamma(s)),\gamma'(s))$ is nonzero for all $s\in I^0$, the function $s\mapsto\rho(\gamma(s))$ is strictly monotonic on $I$. Hence, $\rho(\gamma(s_k))= \rho(x_k)$ implies that
 $(s_k)_{k\in \mathbb N}$ converges and we denote its limit by $\tilde s \in  I$. Note moreover that Proposition \ref{foliation}(2)  implies that 
 $g(\Lambda(\gamma(s_k)),\Lambda(\gamma(s_k)))=g(\Lambda(x_k),\Lambda(x_k))$ for all $k\in\mathbb N$. Hence,
 $g(\Lambda(\gamma(\tilde s)),\Lambda(\gamma(\tilde s)))=g(\Lambda(\tilde x),\Lambda(\tilde x))\neq 0$, which implies $\tilde s\in I^o$ by the assumptions on $\gamma$. 
 Since $\Lambda(\gamma(s))$ is proportional to $\gamma'(s)$ at points $s\in I^o$, there exists a sequence $(t_k)_{k\in \mathbb N}$ 
 converging to $0$ such that $\phi_{t_k}(\gamma(s_k))= \gamma(\tilde s)$. Note that the sequence $(\tilde x_k)_{k\in\mathbb N}:= (\phi_{t_k}(x_k))_{k\in\mathbb N}\in L(\gamma(\tilde s))$ still converges 
 to $\tilde x$. Since the leaf  $L(\gamma(\tilde s))$ is  closed in $M$ by   Proposition \ref{foliation}(3), we must have
 $\tilde x \in L(\gamma(\tilde s))$. Hence, $M'_{\textrm{reg}}$ is closed in $M_{\textrm{reg}}$. Finally, $M'_{\textrm{reg}}=M_{\textrm{reg}}$. 
 
 (2) Since $\gamma(I^o)$ lies on a flow line of $\Lambda$ and $\Lambda$ is nowhere vanishing on $\gamma(I^o)$, the flow $\phi_t$ of $\Lambda$ acts 
 simply transitively on $\gamma (I^o)$ and hence on the set of leaves by (4) of Proposition \ref{foliation} and (1).
 
 (3) 
 Take $s_0\in I^o$ and set $L:=L(\gamma(s_0))$. For any point $x\in M_{\textrm{reg}}$ it follows from (1) that there exist $s\in I^{o}$ such that $x\in L(\gamma(s))$. 
 Since the function $s\mapsto\rho(\gamma(s))$ is strictly monotonic on $I^o$, this $s$ is unique. By (2) there exists a unique $t$ such that 
 $\phi_t(\gamma(s))=\gamma(s_0)\in L$ and set $\Psi(x):=(\phi_t(x), s)\in L\times I^o$.  Then $\Psi: M_{\textrm{reg}}\rightarrow L\times I^o$ is a diffeomorphism 
 as claimed. Indeed, the first two properties are satisfied by construction. 
 In order to show the last  property consider the point $\Psi^{-1}(x,s_2)$ and  the point $y$  of the leaf $L(\Psi^{-1}(x,s_1))$ which is the closest point of the leaf $L(\Psi^{-1}(x,s_1))$ to $\Psi^{-1}(x,s_2)$, which exists by completeness of $g$ and Proposition  \ref{foliation}(3). The minimising geodesic connecting $y$ to $\Psi^{-1}(x,s_2)$ is orthogonal to the leaf  $L(\Psi^{-1}(x,s_1))$. Hence, its velocity vector is proportional to $\Lambda$ at its initial point and so, by Proposition \ref{geodesics1}(1), the image of that geodesic must coincide with a segment of the flow line $\phi_t(y)$. In particular, $y=\Psi^{-1}(x,s_1)$.
 Arguing as at the end of the proof of Proposition \ref{foliation}(4), we obtain that the distance between $ \Psi^{-1}(x,s_1)$ and $\Psi^{-1}(x,s_2)$ equals the  length of $\gamma_{|[s_1,s_2]}$ which in turn is equal to  $s_2-s_1$. 
  
\end{proof}

Now we are ready to give the proof of Theorem \ref{M_compact}:

\begin{proof}[Proof of Theorem \ref{M_compact}]
Fix a point $x_0\in M_0$. It follows from the Hopf--Rinow Theorem that there exists a point $x_1\in M_1$ such that $\textrm{dis}(x_0, M_1)=\textrm{dist}(x_0,x_1)$.
Moreover, completeness of $(M,J, g)$ implies the existence of a minimising geodesic 
$\gamma:[0,1]\rightarrow M$ connecting these points, that is, $\gamma(0)= x_0$ and $\gamma(1)= x_1$. Without 
loss of generality we assume that there is no other point of $M_0$ on $\gamma(s)$ than $x_0$. Without loss of generality we may also assume that $\gamma$ is parameterised by arc-length, since 
otherwise we just multiply the metric $g$ by the appropriate constant to achieve that. 

By (3-4) of Proposition \ref{geodesics1} we know that at any point $s\in[0,1]$ such that $\Lambda(\gamma(s))\neq0$ 
the velocity vector $\gamma'(s)$ is proportional to $\Lambda(\gamma(s))$. We also know that $\Lambda(x_0)=\Lambda(x_1)=0$. 
Let us now show that  $\Lambda(\gamma(s))\neq0$ for all $s\in (0,1)$, which implies that $\gamma$ satisfies the assumptions of Proposition \ref{prop:llast}.
By contradiction, assume there exists $s_0\in(0,1)$ such that $\Lambda$ vanishes at $y:=\gamma(s_0)$. Then also the Killing vector field $K:=J\Lambda$ 
vanishes at $y$. We already remarked that the zero set $N$ of $K$ (being the zero set of a Killing vector field) is a union of connected (totally geodesics closed) 
submanifolds of codimension at least $2$ in $M$. By Lemma \ref{norm_lambda} the property of $K$ to be zero is preserved along the integral curves of the 
$2m + 2\tilde m$-dimensional distribution generated by the direct sum of the  eigenspaces of $A$ corresponding to the eigenvalues $1$ and $0$. 
Hence, $y\in N$ lies in a connected submanifold $N_{st}\subset N$ of $M$ of codimension precisely $2$. Now consider the action of the flow of $K$ on 
the tangent space $T_{y}M$. It acts as the identity on $ T_yN_{st} \subset T_yM$ and therefore by rotations on the $2$-dimensional orthogonal compliment 
of $T_yN_{st}$ in $T_yM$. Hence there exists an isometry generated by the flow that sends the vector $\gamma'(s_0)$ (at the point $y$) to the vector 
$-\gamma'(s_0)$. This in turn implies the existence of one more point of $M_0\cup M_1$ (different from $x_0,x_1$) on the geodesic segment 
$\gamma([0,1])$ (note that $\rho$ is invariant under the flow of $K$, hence the flow can not map $x_0$ to $x_1$ and vice versa), which contradicts our assumption. 
Hence, $\Lambda(\gamma(s))\neq0$ for all $s\in (0,1)$ and $\gamma$ satisfies the assumptions of Proposition \ref{prop:llast} as claimed.

Now, by Proposition \ref{prop:llast}, there exists a diffeomorphism $\Psi: M_{\textrm{reg}}\simeq L(\gamma(\frac{1}{2}))\times(0,1)$ such that
the geodesics tangent to $\Lambda$ correspond to the lines $\{x\} \times (0,1)$ inside $L(\gamma(\tfrac{1}{2}))\times(0,1).$
Then, any geodesic starting orthogonally from a point of $M_0$ reaches in distance $1$ a point of $M_1$. 
Hence, the image $\textrm{exp}(S_{x_0}M_0)$ of the compact set
$$
S_{x_0}M_0:=\{ \xi\in T_{x_0}M  \mid   g(\xi, \xi)= 1\ , \xi\in T^\perp_{x_0}M_0\} 
$$
under the exponential map is contained in $M_1$. In fact, $\exp(S_{x_0}M_0)$ coincides with $M_1$: Let $y_1$ be a point in $M_1$, then, by completeness of $M$, there exists a minimising geodesic $\tilde\gamma:[0,1]\rightarrow M$ connecting $\tilde\gamma(0)=x_0$ and $\tilde\gamma(1)=y_1$. By our assumptions, the length of $\tilde\gamma$ must equal the length of $\gamma$, which is $1$.   Hence, $\tilde\gamma'(0)\in S_{x_0}M_0$ by (4) of Proposition \ref{geodesics1}. Therefore, $\textrm{exp}(S_{x_0}M_0)=M_1$, which in particular implies that $M_1$ is compact.

Now consider the set  
$$
S M_1:=\{ (x,\xi) \in TM  \mid  x \in M_1, \    g(\xi, \xi)\le 1\ , \ \ \xi \in T^{\perp}_{x}M_1\}.  
$$
Since $M_1$ is compact,  $S M_1$ is compact as well. Arguing as above using Proposition \ref {prop:llast} , we conclude that its 
image $\textrm{exp}(S M_1)$ contains all regular points. Since the set of regular points is dense in $M$ and $S M_1$ is compact, we must have 
$\textrm{exp}(S M_1)=M$, which implies that $M$ is compact as claimed. 
\end{proof}

 \section{Proof of Theorem \ref{main_thm} }\label{Aut_group}
 Suppose $(M, J, g)$ is a connected complete K\"ahler manifold of real dimension $2n\geq 4$.
 As explained in Section \ref{structure_section}, it remains to prove Theorem \ref{main_thm} under the assumption that the degree of mobility of $(M,J,g)$ equals $2$. Recall that, by Theorem \ref{Weyl_fund_invariant}, any connected complete K\"ahler manifold with constant positive holomorphic sectional curvature is compact and isometric to $(\CP^n, J, c\textrm{g}_{FS})$, where $c>0$ is some positive constant. Since $(\CP^n, J, c\textrm{g}_{FS})$ has degree of mobility $(n+1)^2>2$ by Remark \ref{deg_on_CP}, we see that the assumption of degree of mobility equal to $2$ implies that $(M,J,g)$ is not of constant positive holomorphic sectional curvature.

 \subsection{The case of degree of mobility 2}
 \label{Mobility_2}
 Suppose $(M,J, g)$ is a K\"ahler manifold of real dimension $2n\geq 4$, denote its Levi-Civita connection by $\nabla$ and consider the induced c-projective manifold $(M,J,[\nabla])$. The c-projective invariance of 
 the metrisability equation \eqref{metrisability_equation} implies that for any c-projective transformation $\phi\in\textrm{CProj}(J,g)$ and any $\eta\in\mathcal S$ we have 
 $\phi^*\eta\in\mathcal S$.
 Moreover, the map
  \begin{align*}
T: \textrm{CProj}(J,g)\times \mathcal S &\rightarrow \mathcal S\\
 (\phi,\eta)\,\,\, \mapsto&\,\,(\phi^{-1})^*\eta
 \end{align*}
 evidentially defines a representation of the group $\textrm{CProj}(J,g)$ on the (finite-dimensional) vector space $\mathcal S$. We set $T_{\phi}:=T(\phi,_-)\in\textrm{GL}(\mathcal S,\R)$ for $\phi\in\textrm{CProj}(J,g)$.
 
 \begin{prop}\label{non-aff_c-proj_trans_mobility2}
 Suppose $(M,J,g)$ is a K\"ahler manifold of dimension $2n\geq4$ with degree of mobility $2$ (i.e.\,$\dim(\mathcal S)=2$). Let $\phi\in\emph{CProj}(J, g)\setminus\emph{Aff}(J, g)$ such that $\det (T_{\phi})>0$.  Then $T_{\phi}\in\emph{GL}(\mathcal S,\R)$ has two distinct positive real eigenvalues.
  \end{prop}
  
 \begin{proof} Let $\phi\in\textrm{CProj}(J, g)\setminus\textrm{Aff}(J, g)$ such that $\det (T_{\phi})>0$. Denote by $\mathcal C_+$ the subset of elements in the $2$-dimensional vector space  $\mathcal S$ that are positive-definite at any point of $M$ and write $\eta\in\mathcal C_+$ for the element in $\mathcal S$ corresponding to $g$. Evidently, $\mathcal C_+$ forms a positive cone in $\mathcal S$, which is preserved by $T_{\phi}$. Note that our assumptions also imply that $T_{\phi}(\eta)$ and $\eta$ are linearly independent elements lying in $\mathcal C_+$ (in particular, $T_{\phi}$ is not a constant multiple of the identity), since otherwise $\phi$ would necessarily be a homothety of $g$ and hence affine. This moreover implies that the cone $\mathcal C_+$ has non-empty interior. Now set $\mathcal C:= \mathcal C_+\cup \mathcal C_-$, where $\mathcal C_-:=-\mathcal C_+$ denotes the cone of elements in $\mathcal S$ that are negative definite at any point of $M$. We claim that the closure $\overline{\mathcal C}=\overline{\mathcal C}_+\cup \overline{\mathcal C}_-$ of $\mathcal C$ does not coincide with $\mathcal S$, which together with the fact that $\mathcal C$ has non-empty interior implies that the boundary of $\mathcal C$ is the union of two distinct lines. Indeed, taking an appropriate linear combination of $\eta$ and $T_{\phi}(\eta)$, one can construct an element $\tilde{\eta}\in\mathcal S$ that at some point of $M$ is indefinite. Note that such an element $\tilde\eta$ can not be the limit of a sequence in $\mathcal C_+$ or $\mathcal C_-$, that is, $\tilde\eta$ can neither be in $\overline{\mathcal C}_+$ nor in $\overline{\mathcal C}_-$, which proves the claim. Since $T_{\phi}$ preserve the cone $\mathcal C_+$, it also preserves its boundary, which we have seen consists of two rays generated by two linearly independent elements in $\mathcal S$. The assumption $\det (T_{\phi})>0$ in addition implies that the two rays of the boundary of $\mathcal C_+$ are preserved individually, which shows that $T_{\phi}$ is diagonalisable with positive eigenvalues. Since, as already observed, $T_{\phi}$ is not a constant multiple of the identity, the claim follows.
\end{proof}

  \begin{lem}\label{Aff=Homotheties}
  Suppose $(M,J, g)$ is a K\"ahler manifold of dimension $2n\geq4$ with degree of mobility $2$ and denote by $\nabla$ the Levi-Civita connection of $g$. 
  If there exists $\tilde\eta\in\mathcal S$ which is not parallel for $\nabla$, then any complex affine transformation is a homothety of $g$. In particular, if there exists $\phi\in\emph{CProj}(J,g)\setminus \emph{Aff}(J,g)$, then any complex affine transformation is a homothety of $g$.
  \end{lem} 
  \begin{proof}
  Let $\eta$ denote the element in $\mathcal S$ corresponding to $g$ and assume $\phi$ is an element of $\textrm{Aff}(J, g)$. 
  Since $\dim(\mathcal S)=2$ and $\{\eta,\tilde\eta\}$ is a basis of $\mathcal S$, we must have
 $\phi^*\eta=c\eta+d\tilde\eta$ for some constant $c,d\in\R$. Differentiating with respect to $\nabla$ yields
 \begin{equation*}
  0=\nabla \phi^*\eta=c\nabla\eta+d\nabla\tilde\eta=d\nabla\tilde\eta,
  \end{equation*}
  which implies $d=0$, since $\nabla\tilde\eta\neq0$ by assumption.
  \end{proof}
  
  \begin{prop}\label{crucial_prop}
  Suppose $(M, J,g)$ is a connected complete K\"ahler manifold of dimension $2n\geq4$ with degree of mobility $2$ and assume that
  $\emph{Aff}(J,g)\subsetneq \emph{CProj}(J,g)$. Then for any $\phi\in\emph{CProj}(J,g)\setminus \emph{Aff}(J,g)$ 
  we must have $\det(T_{\phi})<0$.
  In particular, we have 
  $$\phi,\psi\in\emph{CProj}(J,g)\setminus\emph{Aff}(J,g) \Longrightarrow
  \phi\circ\psi\in\emph{Aff}(J,g),$$ i.e. the index of $\emph{Aff}(J,g)$ in $\emph{CProj}(J,g)$ is two.
  \end{prop}
 Note that Proposition  \ref{crucial_prop} shows that Theorem \ref{main_thm} holds under the assumption of degree of mobility $2$, which is what remains to be shown to establish Theorem \ref{main_thm}.

\subsection{Proof of Proposition \ref{crucial_prop}}\label{Section_Proof_Crucial_Prop}
Throughout this section we suppose that:
\begin{itemize}
\item $(M,J,g)$ is a connected complete K\"ahler manifold of dimension $2n\geq 4$ with Levi-Civita connection $\nabla$,
\item $\textrm{dim}(\mathcal S)=2$,
\item  there exists $\phi\in\textrm{CProj}(J,g)\setminus\textrm{Aff}(J,g)$ such that $\det (T_{\phi}) >0$.
\end{itemize}
Our goal is to show that these assumptions lead to a contradiction, which proves Proposition \ref{crucial_prop}.

By Proposition \ref{non-aff_c-proj_trans_mobility2}, we know that the linear isomorphism $T_{\phi}:\mathcal S \rightarrow\mathcal S$ 
has two distinct positive real eigenvalues, which we denote by $\alpha>\beta>0$. 
Suppose $\eta,\tilde\eta\in\mathcal S$ are eigenvectors of $T_{\phi}$ corresponding 
to $\alpha$ respectively $\beta$. Since $\phi$ is by assumption not affine, the element $g^{-1}\textrm{vol}(g)^{\frac{1}{n+1}}\in\mathcal S$ corresponding to the metric $g$ must be 
a linear combination $g^{-1}\textrm{vol}(g)^{\frac{1}{n+1}}=c\eta+d\tilde\eta$ with $c,d\neq 0$. Hence, by rescaling the eigenvectors $\eta$ and $\tilde\eta$ if necessary, 
we may assume that $$g^{-1}\textrm{vol}(g)^{\frac{1}{n+1}}=\eta+\tilde\eta.$$ We also set $D_{a}{}^{b}=\eta^{cb}g_{ac}\textrm{vol}(g)^{-\frac{1}{n+1}}$ and $\widetilde D_{a}{}^{b}=\tilde\eta^{cb}g_{ac}\textrm{vol}(g)^{-\frac{1}{n+1}}$.
\\\\
Note first that, since $(M,J, g)$ has degree of mobility $2$, the property of a point in $M$ to be regular with respect to an element $B\in\textrm{Sol}(J,g)$, as defined in Definition \ref{reg_points}, 
is invariant under c-projective transformations: If $B$ is a constant multiple of the identity, then any point is regular and so there is nothing to show. Assume now $B$ is not a constant multiple of $\textrm{Id}$ and $x\in M$ a regular point with respect to $B$. Then $x$ is also regular with respect to any other element in $\textrm{Sol}(J,g)$, since any element in $\textrm{Sol}(J,g)$ is a linear combination of $\textrm{Id}$ and $B$.
For any $\psi\in\textrm{CProj}(J,g)$ one has $\psi^*B\in\textrm{Sol}(J,\psi^*g)$, which implies that $x$ is regular with respect to the endomorphism $\psi^*B$, since $\textrm{Sol}(J,\psi^*g)\cong \textrm{Sol}(J,g)$.
Since the eigenvalues of $\psi^*B\in\textrm{Sol}(J,\psi^*g)$ at $x$ coincide with the eigenvalues of $B$ at $\psi(x)$, we deduce that $\psi(x)$ is also regular with respect to $B$.

Let us now fix a regular point $x_0\in M$ with respect to a (hence any) $B\in\textrm{Sol}(J,g)$ linearly independent to $\textrm{Id}$. Then there exists a neighbourhood $U$ of $x_0$ inside the set of regular points and, since $g$ is positive definite, a frame of $TU$, such that $g$ corresponds to the identity matrix and $D$ and $\tilde D$ to diagonal matrices $D=\textrm{diag}(d_1, d_1, ..., d_n, d_n)$ respectively
$\widetilde D=\textrm{diag}(\tilde d_1, \tilde d_1, ..., \tilde d_n, \tilde d_n)$, where $d_i, \tilde d_i$ are smooth real-valued functions on $U$ with $d_i+\tilde d_i=1$ for $i=1,...,n$.
This implies that in the local frame the tensor
$$A(k)_{a}{}^b
:=(\phi^{-k})^*(\ms^{bc}+\sms^{bc})g_{ac}\vol(g)^{-\frac{1}{n+1}}$$
corresponds to the following diagonal matrix:
\begin{equation}\label{A_k}
A(k)=\diag(\alpha^kd_1+\beta^{k}\tilde d_1, \alpha^kd_1+\beta^{k}\tilde d_1,..., \alpha^kd_n+\beta^{k}\tilde d_n, \alpha^kd_n+\beta^{k}\tilde d_n).
\end{equation}

Since $g$ and $(\phi^{-k})^*g$ are positive definite, all diagonal entries of
\eqref{A_k} are positive for all $k\in\Z$. Hence, $d_i+(\frac{\beta}{\alpha})^k\tilde d_i>0$
respectively $(\frac{\alpha}{\beta})^kd_i+\tilde d_i>0$ for all $k$ and taking the limit
$k\rightarrow\infty$ respectively $k\rightarrow -\infty$ shows that
$d_i,\tilde d_i\geq 0$ for all $i=1,\ldots,n$. Since $d_i+\tilde d_i=1$, we
conclude that
\[
0\leq d_i\leq 1\quad \text{ and }\quad 0\leq \tilde d_i\leq 1
\quad\quad \text{ for all } i=1,\ldots,n.
\]

Consider the pull-back $(\phi^{-k})^*(D)$, which 
corresponds in the frame over $U$ to a block diagonal matrix whose $i$-th block is given by
\begin{equation}\label{pull_back_D}
\frac{\alpha^{k}d_i}{\alpha^{k}d_i+\beta^{k}(1-d_i)}\textrm{Id}_2.
\end{equation}
Recall that the eigenvalues of $(\phi^{-k})^*(D)$ at $x_0$ are the same as the eigenvalues of $D$ at $\phi^{-k}(x_0)$. Since 
$\phi^{-k}(x_0)$ is again regular and the multiplicity of a constant eigenvalue of $D$ is constant on the set of regular points by Proposition \ref{prop_eigenvalues}, we therefore 
conclude that the only possible constant eigenvalues of $D$ on $U$ are $0$ and $1$. Since $d_i=0$ (respectively $d_i=1$) on $U$ implies 
$\tilde d_i=1$ (respectively $\tilde d_i=0$), we deduce that the only possible constant eigenvalues of $A(k)$ are $\alpha^k$ and $\beta^k$. Since $U$ 
consists of regular points the distinct eigenvalues of $A(1)$ on $U$ are smooth real-valued functions with constant algebraic multiplicity. We write 
$2m$ respectively $2\tilde m$ for the algebraic multiplicities of the constant eigenvalues $\alpha$ and $\beta$ on $U$. The number of distinct non-constant 
eigenvalues of $A(1)$ is then given by $n-m-\tilde m$. Recall also that the algebraic multiplicities $2m$ and $2\tilde m$ of the constant 
eigenvalues $\alpha$ and $\beta$ of $A(1)$ do not depend on the choice of regular point $x_0$ by Proposition  \ref{prop_eigenvalues}.

\begin{lem}\label{inequ_1}
At $x_0$ and hence at any regular point we must have
\begin{equation}\label{inequalities}
\alpha^{(n-\tilde m)} \leq \beta^{-(\tilde m+1)} \quad\text{and}\quad \alpha^{(m+1)} \geq \beta^{-(n-m)},
 \end{equation}
 which implies 
 \begin{equation}\label{identities}
n-m-\tilde m=1\quad\quad \alpha^{m+1}=\beta^{-(\tilde m+1)}.
\end{equation}
\end{lem}
In \cite{CEMN} the analogue statement (see Lemma 7.7 and the following considerations there) was proved for connected complete K\"ahler manifold $(M,J,g)$ of degree of mobility $2$
under the assumption of the existence of a flow of non-affine c-projective transformations. The proof persists however in our discrete setting. 
For completeness and later purposes we nevertheless give the proof here again. 

\begin{proof}
Note that $(\phi^{-k})^*g=g(G(k)\cdot,\cdot)$, where  $G(k):=\det_\R(A(k))^{-\frac{1}{2}} A(k)^{-1}$.
Without loss of generality we assume that the first $2\ell:=2n-2m-2\tilde m$ diagonal entries of $D$ are not constant (which is equivalent
to assuming that $d_i(x_0)\neq0,1$ for $i=1,\ldots,\ell$), the next $2m$
elements are equal to $1$, and the remaining $2\tilde m$ elements are zero on
$U$.  Then, we deduce from \eqref{A_k} that $G(k)$ on $U$ is a block diagonal
matrix of block sizes $2\ell\times 2\ell$, $2m\times 2m$ and $2\tilde m\times
2\tilde m$ respectively, where the three blocks are given by
\begin{equation}\label{B_k}
\Psi(t)\left(\begin{array}{ccc}
\frac{1}{d_1 \alpha^{k} + (1-d_1) \beta^{k}} \textrm{Id}_2&&\\
&\ddots&\\
&&\frac{1}{d_\ell \alpha^{k} + (1-d_\ell) \beta^{k}}\textrm{Id}_2
\end{array}\right),
\end{equation}
\,\\\\
\begin{equation*}
\Psi(t)\alpha^{-k}\textrm{Id}_{2m},\quad\quad\text{respectively}\quad\quad
\Psi(t) \beta^{-k} \textrm{Id}_{2\tilde m}, 
\end{equation*}
where
\[
\Psi(t):=
\alpha^{-km}\beta^{-\tilde m k}\prod_{i=1}^{\ell}\frac{1}{d_i \alpha^{k} + (1-d_i) \beta^{k}}.
\]
Write $\nu_1,\ldots,\nu_\ell$, $\nu$ and $\tilde\nu$ for the
eigenvalues of these respective diagonal matrices.  Note that their asymptotic
behaviour for $k\to +\infty $ respectively for $k \to -\infty$ is as follows
\begin{equation}\label{asymptotics_eigenvalues}
\begin{array}{cccc}
k\to+\infty & \nu_i(k)\sim \frac{\alpha^{-(n-\tilde m+1)k} \beta^{-\tilde m k} }{d_i\prod d_j}
& \nu(k) \sim \frac{\alpha^{-(n-\tilde m+1)k}\beta^{- \tilde m k} }{\prod d_j}
& \tilde \nu(k)\sim \frac{\alpha^{-(n-\tilde m)k}\beta^{-(\tilde m+1)k} }{\prod d_j}\\ &&&\\
k\to-\infty & \nu_i(k)\sim \frac{\alpha^{-mk}\beta^{-(n-m+1)k} }{(1-d_i)\prod (1-d_j)}
&   \nu(k) \sim \frac{\alpha^{-(m+1)k}\beta^{-(n-m)k} }{\prod (1-d_j)}
&  \tilde \nu(k)\sim \frac{\alpha^{-mk}\beta^{-(n-m+1)k}}{\prod (1- d_j)}.
\end{array}
\end{equation}

Assume now that \eqref{inequalities} is not satisfied. Without loss of generality we can assume the first inequality is not satisfied, that is, we assume that $\alpha^{(n-\tilde m)}> \beta^{-(\tilde m+1)}$, since otherwise we replace $\phi$ by $\phi^{-1}$, 
which exchanges the role of the inequalities in \eqref{inequalities}. 

Now consider the sequence $(\phi^{-k}(x_0))_{k\in\Z_\geq 0}$. Then $\alpha^{(n-\tilde m)}> \beta^{-(\tilde m+1)}$ implies that all eigenvalues of $G(k)$ decay exponentially as $k\to\infty$ by \eqref{asymptotics_eigenvalues}. 
Hence, we conclude that the distance between $\phi^{-k}(x_0)$ 
and $\phi^{-(k+1)}(x_0)$ also decays at least exponentially as $k\to\infty$. This shows that $(\phi^{-k}(x_0))_{k\in\Z_\geq 0}$ is a Cauchy sequence and hence completeness of $(M,J,g)$ implies that it converges. We denote 
the limit of $(\phi^{-k}(x_0))_{k\in\Z_\geq 0}$ by $\tilde x$.

Now let $F$ be the smooth
real-valued function on $M$ given by
\begin{equation*}
F= W_{ab}{}^c{}_d W_{ef}{}^s{}_t
g_{cs} g^{ae} g^{bf} g^{dt},
\end{equation*}
where $W_{ab}{}^c{}_d$ denotes the c-projective Weyl curvature of
$(M, J, [\nabla])$ defined as in \eqref{Weyl_curv}.  
Since $W_{ab}{}^c{}_d$ is
c-projectively invariant, $(\phi^{-k})^*F(x_0)=F(\phi^{-k}(x_0))$
equals
\begin{equation}\label{pull_back_F}
F(\phi^{-k}(x_0))=(W_{ab}{}^c{}_d W_{ef}{}^s{}_t)\,
(\phi^{-k})^*g_{cs}\, (\phi^{-k})^*g^{ae}\,(\phi^{-k})^*g^{bf}\,
(\phi^{-k})^*g^{dt})(x_0).
\end{equation}
Since $F$ is continuous, we have $\lim_{k\to\infty}F(\phi^{-k}(x_0))=F(\tilde x)$. 

In the frame we are working, the matrices corresponding to $g$ and $G(k)$
are diagonal and hence the function $F(\phi^{-k}(x_0))$ is a sum of the form
\begin{equation}\label{sum}
\sum_{1\leq i,j,p,q\leq 2n}C(ijpq;k)
\left( W_{ij}{}^{p}{}_{q}(x_0)  \right)^2, 
\end{equation} 
where the coefficient $C(ijpq;k)$ is the product of the $p$-th diagonal
entry and the reciprocals of the $i$-th, $j$-th and $q$-th diagonal entry
of the diagonal matrix that corresponds to $G(k)$ (and $W_{ij}{}^{p}{}_{q}(x_0)$ denotes the coefficients of $W$ with respect to the frame). The
coefficients $C(ijpq;k)$ depend on $k$ and their asymptotic behaviour for
$k\to\pm\infty$ can be read off from \eqref{asymptotics_eigenvalues}. Note
moreover that all coefficients $C(ijpq;k)$ are positive.

We claim that, if at least one of the indices $i,j$ or $q$ is less or equal
than $2n-2m-2\tilde m$, then
$W_{ij}{}^{p}{}_{q}(x_0)$ vanishes. Indeed,
from \eqref{asymptotics_eigenvalues} we conclude that $(\phi^{-k})^*g$ decays
exponentially at least as $\alpha^{-(n-\tilde m+1)}\beta^ {-\tilde m k}$, which is up
to a constant the smallest eigenvalue of $G(k)$, and that $(\phi^{-k})^*g^{-1}$ goes
exponentially to infinity at least as $\alpha^{(n-\tilde m)k} \beta^{(\tilde m+1)k}$
as $k\to \infty$. Suppose now that at least one of the indices $i, j$ or
$q$ is less or equal than $2n-2m-2\tilde m$. Then we deduce that up to
multiplication by a positive constant $C(ijpq;k)$ behaves asymptotically as
$k\to\infty$ at least as
\[
\alpha^{(n-\tilde m)k} \beta^{(\tilde m+1)k}\alpha^{(n-\tilde m)k} \beta^{(\tilde m+1)k}\alpha^{(n-\tilde m+1)k}\beta^{\tilde mk}\alpha^{-(n-\tilde m+1)k}\beta^{- \tilde mk}
=\alpha^{2(n-\tilde m)k} \beta^{2(\tilde m+1)k}.
\]
Since by assumption $\alpha^{(n-\tilde m)}> \beta^{-(\tilde m+1)}$, we therefore conclude
that the coefficient
\[
C(ijpq;k)\to\infty \quad\quad\text{ as } k\to\infty.
\]
Since all terms in the sum \eqref{sum} are nonnegative and the sequence
$F(\phi^{-k}(x_0))$ converges, we therefore deduce that
$W_{ij}{}^{p}{}_{q}(x_0)=0$ provided that at
least one of the indices $i,j$ or $q$ is less or equal than $2n-2m-2\tilde
m$. Hence, there exist a non-zero vector $V\in T_{x_0}M$ such that 

\[
W_{ab}{}^c{}_d V^a=0,\quad
W_{ab}{}^c{}_d V^b=0\quad\text{and}\quad
W_{ab}{}^c{}_d V^d=0, 
\]
which shows that $(3)$ of Remark 6.3 of \cite{CEMN} is
satisfied, which implies that $(M,J,g)$ has so-called nullity at $x_0$ and hence on the set of regular points, since $x_0$ was arbitrary.
Now Theorem 7.2 of \cite{CEMN} says that, if $(M,J, g)$ is a connected complete K\"ahler manifold with nullity on a dense open set and whose 
holomorphic sectional curvature is not a positive constant (which is implied by our assumption that the degree of mobility is $2$), then any complete K\"ahler metric that is c-projectively equivalent to $g$ is 
actually affinely equivalent to $g$, which contradicts our assumption that $\phi$ is not affine. Hence, the inequalities \eqref{inequalities} 
must be satisfied. Now dividing the first inequality by the second shows $\alpha^{n-\tilde m-m-1}\leq \beta^{n-m-\tilde m-1}$, which implies that $n-\tilde m-m=1$, 
since $\alpha>\beta>0$ by assumption. Hence, $n-\tilde m=m+1$ and inserting this back into the \eqref{inequalities} shows that $\alpha^{m+1}=\beta^{-(\tilde m+1)}$ as claimed.
\end{proof}

To prove Proposition \ref{crucial_prop} it remains to show that also \eqref{identities} leads to a contradiction. Here, we can not proceed as in \cite{CEMN}, where the analogue statement 
was ruled out under the assumption of the existence a flow of non-affine  c-projective transformations. Our strategy will be instead to show that under assumption \eqref{identities} the 
K\"ahler manifold $(M,J, g)$ is necessarily compact. Then a similar reasoning as in the proof of Lemma \ref{inequ_1} shows:

\begin{lem}\label{identities_not_satisfied} 
If $M$ is compact, then \eqref{identities} is not satisfied on the set of regular point of $M$.
\end{lem}

\begin{proof} Assume that the identities \eqref{identities} are satisfied and  fix a regular point $x_0\in M$.
As in the proof of Lemma \ref{inequ_1} consider the sequence  $(\phi^{-k})^*g=g(G(k)\cdot,\cdot)$. 
Note that the identities \eqref{identities} now imply that the asymptotic behaviour \eqref{asymptotics_eigenvalues} of the eigenvalues of $G(k)$ reads as follows: 
 \begin{equation} \label{asymptotics_eigenvalues2}
 \begin{array}{c c c c}
 k\to + \infty & \nu_1(k)\sim  \tfrac{1}{d_1^2} \left(\frac{\beta}{\alpha}\right)^k    & \nu(k) \sim 
 \tfrac{1}{d_1 } \left(\tfrac{\beta}{\alpha}\right)^k & \tilde \nu(k)\sim \tfrac{1}{d_1 }  \\ &&&\\
   k\to - \infty & \nu_1(k)\sim  \tfrac{1}{(1-d_1)^2} \left(\tfrac{\alpha}{\beta}\right)^k  &   \nu(k) \sim 
 \tfrac{1}{(1-d_1)}   &  \tilde\nu(k)\sim \tfrac{1}{(1-d_1)} \left(\tfrac{\alpha}{\beta}\right)^k .
 \end{array}
 \end{equation}
Now consider again $(\phi^{-k}(x_0))_{k\in\Z_\geq 0}$. Note that in contrast to the reasoning in the proof of Lemma \ref{inequ_1} the asymptotics \eqref{asymptotics_eigenvalues2} do not allow us to conclude that the sequence $(\phi^{-k}(x_0))_{k\in\Z_\geq 0}$ is a Cauchy sequence (in fact there also counter examples). Since $M$ is compact, there is however a subsequence $(x_{\ell})_{\ell\in\Z}:=(\phi^{-{k_\ell}}(x_0))_{\ell\in\Z}$ of $(\phi^{-k}(x_0))_{k\in\Z}$  that converges as $\ell\to\pm\infty$, where $k_{-\ell}=-k_\ell$. This implies that
$F(x_\ell)$, which is given as in \eqref{sum} by
\begin{equation}\label{sum2}
\sum_{1\leq i,j,p,q\leq 2n}C(ijpq;k_{\ell})
\left( W_{ij}{}^{p}{}_{q}(x_0)  \right)^2, 
\end{equation} 
converges as $\ell\to\pm\infty$. The coefficient $C(ijpq;k_{\ell})$ is the product of the $p$-th and the reciprocals of the 
$i$-th, $j$-th and $q$-th eigenvalues of $G(k_{\ell})$ and from \eqref{asymptotics_eigenvalues2} we deduce that
any such product either goes to $\infty$ or $-\infty$ as $\ell$ goes to $\infty$ respectively $-\infty$. Hence, 
$W_{ab}{}^c{}_d$ vanishes at $x_0$. Since $x_0$ was an arbitrary regular point, it vanishes on the open dense subset of regular points of $M$ 
and hence everywhere by connectedness of $M$. Hence, $(M,J, g)$ is c-projectively flat. Thus, the above mentioned Remark 6.3 and Theorem 7.2 of \cite{CEMN}
imply again that either $g$ has constant positive holomorphic sectional curvature or $\phi$ is affine, which contradict our assumptions.
 \end{proof}
 
Hence, if we can show that $M$ is compact, then Lemma \ref{inequ_1} and Lemma \ref{identities_not_satisfied} lead to a contradiction, which completes the proof of Proposition 
\ref{crucial_prop}.

In order to show that $M$ is compact, note that \eqref{identities} implies that $D$ locally around any regular point (with respect to some adapted local frame) is of the form

\begin{equation}
D= \left(\begin{array}{ccc}
 \rho \textrm{Id}_{2}&&\\
& \textrm{Id}_{2m}&\\
&& 0_{2\tilde m}
\end{array}\right),
\end{equation}
where $\rho(x)\in(0,1)$ for all $x\in M_{\textrm{reg}}$. Hence, $D$ satisfies property (P) as in Section \ref{special_sol} and $\rho$ extends to a smooth function on all of $M$ with values in $[0,1]$. 
As in Section \ref{special_sol} we set again $M_i:=\{x\in M: \rho(x)=i\}$ for $i=0,1$ and denote by $\Lambda$ the gradient of $\rho$ with respect to $g$.

\begin{lem} \label{M_i_non-empty}
If \eqref{identities} holds, then $$M_i:=\{x\in M: \rho(x)=i\}\neq \emptyset\quad\quad \textrm{ for } i=0,1,$$ where $\rho=\frac{1}{2}D_{a}{}^a-m$.
\end{lem}
\begin{proof}
Let $\gamma: \R\rightarrow M$ be any arc-length parametrised geodesic such that $x_0:=\gamma(0)\in M_{\textrm{reg}}$ and $\gamma'(0)$ is proportional to $\Lambda(\gamma(0))$ with a positive coefficient.
Now let $I\subseteq \R$ with $0\in I$ such that $\gamma|_I$ satisfies the assumptions $(a-c)$ of Proposition \ref{prop:llast}. Since $\phi$ preserves the set of regular points, the point $\phi^{-1}(\gamma(0))$ 
is again a regular point. Hence, Proposition \ref{prop:llast} implies that there exists a unique $s_1\in I^o$ such that $\phi^{-1}(x_0)\in L(\gamma(s_1))$. Iterating this procedure, we obtain a sequence 
$(s_k)_{k\in\Z},$ where $s_0:=0$ and $s_k\in I^o$ is inductively defined by $\phi(\gamma(s_k))\in L(\gamma(s_{k+1}))$. By construction, we have $\rho(\gamma(s_k))=\rho(\phi^{-k}(x_0))$.
From \eqref{pull_back_D} we conclude that
\begin{equation}
\rho(\gamma(s_k))=\rho(\phi^{-k}(x_0))=\frac{\alpha^k\rho(x_0)}{\alpha^k\rho(x_0)+\beta^k(1-\rho(x_0))}.
\end{equation}
Hence, \eqref{identities} implies that $\rho(\gamma(s_k))$ is strictly increasing in $k$ 
and $\lim_{k\to\infty}\rho(\gamma(s_k))=1$ respectively $\lim_{k\to-\infty}\rho(\gamma(s_k))=0$.
By assumption we moreover have $g(\gamma'(s), \Lambda(\gamma(s)))>0$ on $I^o$, which implies that $s\mapsto \rho(\gamma(s))$ is 
strictly increasing on $I$. Hence, $(s_k)_{k\in\Z}\in I^o$ must be also strictly increasing. 
From Proposition \ref{prop:llast} and the asymptotics \eqref{asymptotics_eigenvalues2}, the length of the segment $\gamma_{[s_k, s_{k+1}]}$, which equals $s_{k+1}-s_k$, 
behaves asymptotically (up to a constant) as $\frac{\beta^k}{\alpha^k}$, which converges to $0$ for $k\to\infty$. Hence, the limits 
$s_{\pm}:=\lim_{k\to\pm\infty}s_k$ exist and $\rho(\gamma(s_+))=\lim_{k\to\infty}\rho(\gamma(s_k))=1$ respectively 
$\rho(\gamma(s_-))=\lim_{k\to-\infty}\rho(\gamma(s_k))=0$ imply $M_i\neq \emptyset$ for $i=0,1$.
\end{proof}

By Lemma \ref{M_i_non-empty} and Theorem \ref{M_compact}, we deduce that $M$ is compact. By Lemma \ref{inequ_1} this is however a contradiction.
Therefore, Proposition \ref{crucial_prop} holds, and, as a consequence, Theorem \ref{main_thm} is proved. 

\section{Proof of Theorem \ref{main_thm2}}\label{proof_thm2}
Suppose that $(M,J, g)$ is a complete connected K\"ahler manifold satisfying the assumptions of Theorem \ref{main_thm2}.
Since $g$ is not of constant positive holomorphic sectional curvature and $\textrm{Aff}(J,g)\subsetneq \textrm{CProj}(J,g)$, the degree of mobility of $(M,J, g)$ must be two by Theorem \ref{deg_at_least_3} and 
the index of $\textrm{Aff}(J,g)$ in $\textrm{CProj}(J,g)$ must be two by Theorem \ref{main_thm}. By Lemma \ref{Aff=Homotheties} moreover, the group $\textrm{Aff}(J,g)$ must equal the group of homotheties of $g$.
Hence, Theorem \ref{main_thm2} follows from the fact that on a locally non-flat connected complete Riemannian manifold the group of homotheties coincides with the isometry group of $g$ \cite[Lemma 2]{IO}.

 \end{document}